\documentclass{article}

\usepackage{amsmath, amssymb, amsthm}
\usepackage{graphicx}
\usepackage[hidelinks]{hyperref}
\usepackage{algorithm, algpseudocode}
\usepackage{tikz}
\usepackage{myArticle}
\usepackage{myTools}
\newcommand{\bs}{\boldsymbol}
\usepackage{booktabs}
\usepackage{fancyhdr}
\usepackage{titlesec}

\title{Decision-Aware Predictions for Right-Hand Side Parameters in Linear Programs}
\newcommand{\shorttitle}{Decision-Aware Right-Hand Side Predictions}

\author{
Jackson Forner\thanks{jforner@smu.edu}, Miju Ahn\thanks{mijua@smu.edu}, and Harsha Gangammanavar\thanks{harsha@smu.edu} \\
Department of Operations Research and Engineering Management, \\ Southern Methodist University, Dallas, TX
}
\newcommand{\authorshort}{Forner, Ahn, Gangammanavar}

\date{First submission: November 30, 2025; Second submission: March 15, 2026}

\begin{document}

\maketitle

\begin{abstract}
This paper studies an integrated learning and optimization problem in which a prediction model estimates the right-hand-side parameters of a linear program (LP) using a contextual vector. Considering that such a prediction alters the feasible region of the LP, we aim to estimate the constraint set to contain the optimal solution of the underlying LP, given by the true right-hand side parameters. We propose formulations for training a prediction model by minimizing the decision error while accounting for feasibility, measured by a collection of historical primal and dual solutions. Our analysis identifies conditions under which a resulting predicted feasible region contains the true solution, and whether the latter solution achieves optimality for the predicted problem. To solve the alternative training problems, we employ existing LP and nonconvex programming solution methods. We conduct numerical experiments on a synthetic LP and a network optimization problem. 
Our results indicate that the proposed methods effectively implement the desired feasibility, compared to standard regression models.
\end{abstract}

\section{Introduction} \label{sec:intro}
In this paper, we consider a contextual linear programming (C-LP) problem of the following form:
\begin{equation}\label{prob:downstream}
    \min_{\bs{x}}~\big \{\inner{\bs{c}, \bs{x}} ~|~
    	\bs{Ax} \geq \bs{b}(\bs{\obs}),~ \bs{x}\geq \bs{0}\big \}.
\end{equation}
We seek an optimal solution to an instance of the above problem in the feasible region $\set{X}(\tilde{\bs{b}}) \coloneqq \{\bs{x} \geq \bs{0} ~ | ~ \bs{Ax} \geq \tilde{\bs{b}}\} \subseteq \RR^n$ that is parametrized by a deterministic cost coefficient $\bs{c} \in \RR^n$,  a deterministic constraint matrix $\bs{A} \in \RR^{m \times n}$, and a realization $\bs{b}(\bullet) \in \RR^m$ of a stochastic right-hand side $\tilde{\bs{b}}$. The stochastic right-hand side is correlated with a context, or a feature, vector that we denote by $\tilde{\bs{\obs}} \in \RR^d$. In other words, a joint probability distribution links the context vector to the problem's right-hand-side parameter. We consider a setting where we only observe a realization $\bs{\obs}$ of the feature $\tilde{\bs{\obs}}$ before the decision epoch and must determine a decision using a prediction $\hat{\bs{b}}(\obs)$ of the right-hand side vector. In this paper, we study different approaches to make such predictions.

To handle the contextual problems, of which \eqref{prob:downstream} is a particular form, integrated learning and optimization (ILO) is a particularly compelling approach. In this approach, models are trained to predict uncertain optimization parameters in a way that minimizes the error in the decisions made based on those predictions, rather than the error in the parameter predictions themselves. To the best of our knowledge, the earliest work on integrated learning and optimization is by \cite{bengio1997using}, who considered a time-series problem in portfolio selection. More recently, \cite{elmachtoub2022smart} considered linear programs with uncertain cost vectors and proposed novel loss functions for predicting these parameters, directly incorporating the optimization problem's structure into the learning process. Their approach created a new ``smart-predict-then-optimize" (SPO) framework that has since led to other similar works in recent years. {For example, \cite{mandi2020smart} and \cite{demirovic2020dynamic} investigated the SPO framework for predicting objective function coefficients in combinatorial optimization problems. Importantly, all of these works utilize a notion of regret defined as the suboptimality of a solution induced by a parameter prediction.} We refer interested readers to the work of \cite{sadana2024survey} for a more comprehensive review of ILO and, more broadly, contextual optimization. It is worth noting that most works on ILO assume the constraints to be deterministic and, therefore, do not apply to \eqref{prob:downstream}, where the constraints have stochastic right-hand sides.

{When the constraints are deterministic, feasibility is not a concern. On the other hand, when estimating constraint parameters, even slight errors in estimation can turn a nonempty feasible region into an empty one. Moreover, the true optimal solution may become infeasible under the predicted constraints, or conversely, the optimal solution of the predicted program may become infeasible under the true constraints. Despite these challenges, there have been a few works in the ILO framework that focus on predicting constraint parameters.}
For example, \cite{hu2023two} considered mixed-integer linear programs with uncertain parameters in the objective and constraints, and trained a neural network to estimate these values using a post-hoc regret loss function similar to that of \cite{elmachtoub2022smart}.
\cite{estes2023smart} also applied a regret-type loss function to estimate right-hand side parameters, which are in the second stage of a two-stage stochastic {linear} program.
{Even though these works have to do with the prediction of constraint parameters, they utilize loss functions where regret is defined in a similar way as for objective function coefficient estimation, namely, suboptimality of the optimal solution of the predicted program in the true program (note that \cite{hu2023two} also include a penalty for modifying a softly-commited stage 1 solution to a final stage 2 solution in their regret loss function). In contrast to this, we minimize a loss function which measures the duality gap of the true solution pair $(\bs{x}^\star, \bs{y}^\star)$ with respect to the predicted C-LP while enforcing that the predicted C-LP is feasible by making the true optimal solution a feasible solution.
In terms of solution approach, our setting is more similar to that of \cite{estes2023smart} since we train in an offline setting in a single step using a dataset of historical observations, unlike the work of \cite{hu2023two} which considers a two-stage approach to training a predictor}.

Our work extends the ILO framework to parameters in the constraints of optimization problems. In this regard, the main contribution of this paper is twofold.
\begin{enumerate}
    \item Errors in predicting constraint parameters may render the optimization problem infeasible. To address this, we present four different training problems that explicitly account for loss/decision errors, measured in terms of feasibility and suboptimality. These models use a training dataset comprising true or historical data, context vectors, right-hand-side vectors, and optimal solutions. The models differ in how they utilize the data. We identify the conditions under which we can recover the true optimal solutions as feasible or optimal solutions of the optimization problem with predicted parameters. We also present solution methods to solve the alternative training problems. 
    \item We validate the proposed training problems through numerical experiments conducted on a synthetic LP and a network optimization problem. We compare the feasibility and suboptimality metrics attained by predictors obtained using the alternate training problems, and also benchmark them against standard training approaches that do not account for decision errors. Our results indicate that our proposed models are able to leverage decision data to achieve high feasibility in terms of containing the true optimal solution when we explicitly enforce such constraints in the learning process, and that the feasibility of such models increases as we train on more data. The benchmark models, on the other hand, do not leverage decision data and attain very low feasibility. We also observe a trade-off between feasibility and suboptimality in our proposed models, namely, models that attain a higher feasibility perform worse in terms of suboptimality, and vice versa.
\end{enumerate}

The remainder of the paper is structured as follows: in \S\ref{sec:the_framework}, we present a framework identifying various goals for the problem of right-hand side parameter predictions in LPs. We then propose a set of novel learning problems aimed at achieving these goals and identify suitable algorithms to solve them. In \S\ref{sec:numerical_experiments}, we present numerical experiments using a synthetic and a network optimization problem that demonstrate the predictive utility of our proposed learning problems in terms of relevant decision metrics. We present all of the proofs and additional details in the Appendix.

\subsubsection*{Notations} Let $[N] \coloneqq \{1, \dots, N\}$. We define a collection of vectors as $(\bs{x}_i) \coloneqq (\bs{x}_1, \dots, \bs{x}_N)$. For a matrix $\bs{A} \in \mathbb{R}^{m \times n}$, we denote its $j$-th row as a vector $\bs{a}_j \in \mathbb{R}^{n}$.

\section{The Framework}\label{sec:the_framework}
We consider a setting where the goal is to identify an optimal solution to the C-LP problem \eqref{prob:downstream} using only an observation of the context vector $\bs{\obs}$. We denote the optimal primal-dual solution pair of \eqref{prob:downstream} with an arbitrary right-hand side vector $\bs{b}$ by $(\bs{x}^\star, \bs{y}^\star)$, and the associated optimal objective function value by $v^\star$. Notice that the optimal solutions and the values are functions of the right-hand side, but we suppress the explicit notation (e.g., $\bs{x}^\star(\bs{b})$) for notational convenience. To model this decision-making setting, we define a probability space $(\rvset \times \set{B}, \set{F}, \PP)$, where $\rvset \subset \RR^d$ is a compact set, $\set{B} \subseteq \RR^m$, $\set{F}$ is a $\sigma$-algebra over $\rvset \times \set{B}$, and $\PP$ is a joint probability distribution over $\rvset \times \set{B}$. For each $\bs{\obs} \in \rvset$, we assume that the optimal cost $v^\star$ induced by the corresponding right-hand side vector $\bs{b}$ is finite; that is, the problem instance corresponding to any observation of the context vector $\bs{\rv}$ is feasible and has a finite optimal cost. In our setting, we consider a collection of independent observations of the context and the right-hand side vectors, i.e., $\set{D}_N \coloneqq \{(\bs{\obs}_i, \bs{b}_i)\}_{i \in [N]}$ from $\rvset \times \set{B}$. These observations are either based on historical data or are generated using a simulation process. For any observation, we instantiate \eqref{prob:downstream} with $\bs{b_i}$ as the right-hand side of the constraints and solve it to optimality. We denote the resulting optimal primal-dual solution pair by $(\bs{x}_{i}^{\star}, \bs{y}_{i}^{\star})$, and the associated optimal value as $v_i^\star$. Using this optimal solution data, we denote a decision-induced dataset by $\set{D}_N^\star \coloneqq \{(\bs{\obs}_i, \bs{b}_i, \bs{x}_{i}^{\star}, \bs{y}_{i}^{\star})\}_{i \in [N]}$.

We consider a class $\set{P}$ of predictors of the right-hand side vector $p: \rvset \rightarrow \set{B}$. We denote $\hat{\bs{b}} \coloneqq p(\bs{\obs})$ as the predicted right-hand side vector corresponding to context vector $\bs{\obs}$. We can use the prediction $\hat{\bs{b}}$ to instantiate the C-LP problem \eqref{prob:downstream} to obtain the ``predicted problem.'' We denote the optimal primal-dual solution pair obtained by solving the predicted problem by $(\hat{\bs{x}}, \hat{\bs{y}})$ and the corresponding optimal objective function value $\hat{v}$, assuming they exist. We denote by $\ell(\cdot, \cdot)$ a loss function that measures, upon observation of the true right-hand vector, the error incurred when we use a prediction $\hat{\bs{b}}$ in lieu of the true right-hand side $\bs{b}$. As is customary in machine learning, we utilize $\set{D}_n$ (or $\set{D}_N^\star$) as the training data to identify the prediction model $p^\star \in \set{P}$ by solving the empirical risk minimization (ERM) problem:
\begin{equation} \label{prob:erm}
    \min_{p \in \set{P}} \bigg \{\frac{1}{N} \sum_{i \in [N]} \ell(p(\bs{\obs}_i), \bs{b}_i) \bigg \}.
\end{equation}
To evaluate the quality of the model $p^{\star}$ obtained from solving \eqref{prob:erm}, we utilize a validation dataset as $\set{V} \coloneqq \{(\bs{\obs}_i^v, \bs{b}_i^v)\}$. We denote the optimal primal-dual solution pair obtained by solving the predicted problem with $p^\star(\bs{\obs}_{i}^v)$ by $(\hat{\bs{x}}_i, \hat{\bs{y}}_i)$ and the corresponding optimal objective function value $\hat{v}_i$. In this paper, we focus on the class of linear prediction models $\set{P} = \{p~|~\exists \bs{W} \in \mathbb{R}^{m \times d} ~\text{s.t.}~p(\bs{\obs}) = \bs{W\obs}, \forall \bs{\obs} \in \rvset\}$, in which case the ERM problem \eqref{prob:erm} reduces to an optimization over prediction matrix $\bs{W}$.

To design an appropriate loss function, we consider the following set that captures the primal and dual feasible solutions of the predicted problem: 
\begin{equation} \label{eq:feasSet}
    \widehat{\set{S}}(\bs{\obs}; p) \coloneqq \left \{ (\bs{x},\bs{y})\left \vert 
    \begin{array}{l} 
        \bs{Ax} \geq p(\bs{\obs}),~ \bs{x} \geq \bs{0},\\ \bs{A}^\top \bs{y} \leq \bs{c},~ \bs{y} \geq \bs{0}
    \end{array}
    \right. \right \}.
\end{equation}
In the above, $\bs{y} \in \RR^m$ is the dual variable, an element of the LP dual feasible region given by $\set{Y} \coloneqq \{\bs{y} \geq \bs{0} ~|~\bs{A}^\top \bs{y} \leq \bs{c}\}$. We denote by $\widehat{\set{S}}^{\star}(\bs{\obs}; p)$ a refinement of the above set to include the first-order optimality conditions of the predicted problem. That is,
\begin{equation} \label{eq:optimalSolnSet}
    \widehat{\set{S}}^{\star}(\bs{\obs}; p) \coloneqq \left\{(\bs{x},\bs{y}) \in S(\bs{\obs}; p) ~ | ~ \inner{\bs{c},\bs{x}} = \inner{p(\bs{\obs}), \bs{y}}\right\}.
\end{equation}
In our setting, when we observe a new context vector $\bs{\obs}$, we predict the right-hand side as $\hat{\bs{b}} = p(\bs{\obs})$ and instantiate the C-LP \eqref{prob:downstream}. We anticipate that the optimal primal or dual solution of the true C-LP corresponding to the unobserved right-hand side $\bs{b}$ at least resides in the feasible region of the predicted problem; that is, there exists $\bs{y} \in \set{Y}$ such that $(\bs{x}^\star, \bs{y}) \in \widehat{\set{S}}(\bs{\obs}; p)$, or there exists $\bs{x} \in \set{X}(\bs{b})$ such that $(\bs{x}, \bs{y}^\star) \in \widehat{\set{S}}(\bs{\obs}; p)$. Better yet, we may hope for $(\bs{x}^\star, \bs{y}^\star) \in \widehat{\set{S}}(\bs{\obs}; p)$. The best case outcome is that $(\bs{x}^\star, \bs{y}^\star) \in \widehat{\set{S}}^{\star}(\bs{\obs}; p)$, implying that the true optimal solution pair is also optimal for the predicted problem.

\subsection{Different Approaches to Train a Predictor} \label{sec:diff_approaches}
Our ability to realize the minimal or optimistic expectations depends on how well we learn the model $p \in \set{P}$. For this task, we present a suite of training problems that utilize the decision-induced dataset $\set{D}_N^\star$ (referring to the literature, we may describe these learning problems as being decision-aware). In all our training problems, we aim to minimize a metric that can be interpreted as the duality gap, where the constraints capture our expectations identified in the definition of sets $\widehat{\set{S}}(\bs{\obs}; p)$ and $\widehat{\set{S}}^\star(\bs{\obs}; p)$. 

The first training problem in this suite directly targets the optimistic goal of $(\bs{x}^\star, \bs{y}^\star) \in \widehat{\set{S}}^{\star}(\bs{\obs}; p)$. Following \eqref{eq:optimalSolnSet}, we state this optimistic decision-aware learning (DAL) problem as
\begin{equation}\label{prob:optimisticTrain}
    \min_{p \in \set{P}} \bigg \{\frac{1}{N} \sum_{i \in [N]} (\inner{\bs{c} , \bs{x}_{i}^{\star}} -  \inner{p(\bs{\obs}_i), \bs{y}_{i}^{\star}}) ~\bigg \vert~ \bs{A}\bs{x}_{i}^{\star} \geq p(\bs{\obs}_{i}) \quad \forall i \in [N]\bigg \}.
\end{equation}
Notice that since $(\bs{x}_{i}^{\star}, \bs{y}_{i}^{\star})$ are optimal solution pairs to the true problem, they satisfy $\bs{x}_{i}^{\star}, \bs{y}_{i}^{\star} \geq \bs{0}$ and $\bs{A}^\top \bs{y}_{i}^{\star} \leq \bs{c}$. The additional constraint in \eqref{prob:optimisticTrain} ensures the feasibility of $\bs{x}^\star$ to the predicted problem. Notice that each summand in the above problem is nonnegative since the pair $\bs{x}_{i}^{\star}$ and $\bs{y}_{i}^{\star}$ are feasible to the predicted primal and dual problems, respectively. Moreover, this problem can be reformulated as an LP problem if the model $p$ is a linear model, and if its optimal value is zero, then it implies that $(\bs{x}_{i}^{\star}, \bs{y}_{i}^{\star}) \in \widehat{\set{S}}^{\star}(\bs{\obs}_i; p)$ for all $i \in [N]$. However, such an outcome may be unlikely.

Alternatively, if our goal is to at least recover the true primal optimal solutions from the predicted problems, then we can consider a primal-DAL training problem stated as
\begin{equation} \label{prob:fixedPrimalTrain}
    \min_{p \in \set{P}, (\bs{y}_i)} \bigg \{\frac{1}{N} \sum_{i \in [N]} (\inner{\bs{c} , \bs{x}_{i}^{\star}} -  \inner{p(\bs{\obs}_i), \bs{y}_i}) ~\bigg \vert~ \bs{A}\bs{x}_{i}^{\star} \geq p(\bs{\obs}_{i}),~\bs{A}^\top \bs{y}_i \leq \bs{c},~ \bs{y}_i \geq \bs{0} \quad \forall i \in [N]\bigg \}.
\end{equation}
Here, we insist that the true primal solutions reside in the primal feasible region of their corresponding predicted problem. In addition to the model $p$, we also determine the dual variables $\bs{y}_i$, which are required to satisfy the dual feasibility condition for each $i \in [N]$.

Since the dual feasibility requirements are imposed for every data point separately in \eqref{prob:fixedPrimalTrain}, it is possible that the above optimization problem chooses a weak model that satisfies $\bs{A}\bs{x}_{i}^{\star} \geq p(\bs{\obs})$ and still achieves a near-zero objective. To address this issue, we present a slight revision to the above problem:
\begin{equation} \label{prob:fixedPrimalBoundTrain}
    \min_{p \in \set{P}, (\bs{y}_i)} \bigg \{\frac{1}{N} \sum_{i \in [N]} (\inner{\bs{c} , \bs{x}_{i}^{\star}} -  \inner{p(\bs{\obs}_i), \bs{y}_i}) ~\bigg \vert~ \bs{A}\bs{x}_{i}^{\star} \geq p(\bs{\obs}_{i}) \geq \bs{b}_i,~\bs{A}^\top \bs{y}_i \leq \bs{c},~ \bs{y}_i \geq \bs{0} \quad \forall i \in [N]\bigg \}.
\end{equation}
Using the fact that $\bs{A}\bs{x}_{i}^{\star} \geq \bs{b}_i$, here we impose an additional restriction on the model as $p(\bs{\obs}_i) \geq \bs{b}_i$. 

If our goal is to recover the true dual solutions from the predicted problem, then we pose the following dual-DAL training problem:
\begin{equation} \label{prob:fixedDualTrain}
    \min_{p \in \set{P}, (\bs{x}_i)} \bigg \{\frac{1}{N} \sum_{i \in [N]} (\inner{\bs{c} , \bs{x}_i} -  \inner{p(\bs{\obs}_i), \bs{y}_{i}^{\star}}) ~\bigg \vert~ \bs{A}\bs{x}_{i} \geq p(\bs{\obs}_i), ~\bs{x}_i \geq \bs{0} \quad \forall i \in [N]\bigg \}.
\end{equation}
Notice that the above problem has a trivial solution, rendering it useless.

{
    Firstly, notice that the proposed training problems require different historical decision data that the optimizer may or may not have access to, i.e., optimal primal solutions $\bs{x}_i^\star$, optimal dual solutions $\bs{y}_i^\star$, or both. For example, if there are no historical optimal dual solutions $\bs{y}_i^\star$ in the dataset (and perhaps no easy way to obtain them), then one may only use formulation \eqref{prob:fixedPrimalBoundTrain}. Secondly, we rely on training data that comprises optimal solutions of programs that were previously solved. Even if the previously solved programs were degenerate, resulting in multiple primal or dual solutions, the solvers often report only of the optimal solutions. Therefore, our training dataset includes only a single primal-dual pair in every datapoint and do not account for degeneracy in the training data. Nevertheless, the impact of on the training problems is a worthwhile future research direction.}

\subsection{A Discussion on Recovering $(\bs{x}^\star, \bs{y}^\star)$}
Consider an arbitrary pair $(\bs{\obs},\bs{b})$ and associated $(\bs{x}^{\star}, \bs{y}^{\star})$. We are interested in whether $p(\bs{\obs}) = \hat{\bs{b}}$ yields a feasible region that recovers the pair of optimal solutions, i.e., $(\bs{x}^{\star}, \bs{y}^{\star}) \in \widehat{\set{S}}(\bs{\obs}; p)$. It is not difficult to verify that if the model underpredicts, that is, $\bs{b} \geq \hat{\bs{b}}$, then we have such a recovery. However, if the model overpredicts an index that belongs to a subset of indices, such an inclusion relationship does not hold. Proposition~\ref{prop:recoveryS} formally states these observations. For this purpose, we define the following sets of indices:
\begin{equation*}
    \set{J}^{=}(\bs{x}^{\star}) \coloneqq \{j \in [m] ~ | ~ \inner{\bs{a}_{j}, \bs{x}^{\star}} = b_{j} \} 
    \qquad \text{and} \qquad
    \set{J}^{+}(\bs{y}^{\star}) \coloneqq \{j \in [m] ~ | ~ y^{\star}_{j} > 0\}.
\end{equation*}
Among the two sets, we have $\set{J}^{+}(\bs{y}^{\star}) \subseteq \set{J}^{=}(\bs{x}^{\star})$ due to the complementary slackness condition of a linear program. For more meaningful analysis, we assume $\bs{y}^{\star} \neq \bs{0}$, i.e., $\set{J}^{+}(\bs{y}^{\star}) \neq \emptyset$.

\begin{proposition} \label{prop:recoveryS}
    Consider an arbitrary quadruple $(\bs{\obs}, \bs{b}, \bs{x}^{\star}, \bs{y}^{\star})$. Let $p(\bs{\obs}) = \hat{\bs{b}}$. The following holds:
    \begin{enumerate}
        \item[(i)] If $\bs{b} \geq  \hat{\bs{b}}$, then $(\bs{x}^{\star}, \bs{y}^{\star}) \in \widehat{\set{S}}(\bs{\obs}; p)$.
        \item[(ii)] If $\exists \ j' \in \set{J}^{=}(\bs{x}^{\star})$ such that $b_{j'} < \hat{b}_{j'}$, then $(\bs{x}^{\star}, \bs{y}^{\star}) \notin \widehat{\set{S}}(\bs{\obs}; p)$.
    \end{enumerate}
\end{proposition}

The proofs of all the results shown in this paper are presented in Appendix \S\ref{sec:appendix_proof}. 
We note that the contrapositive of the second statement of Proposition~\ref{prop:recoveryS} also serves as a necessary condition for $(\bs{x}^{\star}, \bs{y}^{\star}) \in \widehat{\set{S}}(\bs{\obs}; p)$. In other words, $\set{J}^=(\bs{x}^\star)$ is the smallest index set for which the overprediction of a component $b_j$ yields $(\bs{x}^{\star}, \bs{y}^{\star}) \notin \widehat{\set{S}}(\bs{\obs}; p)$. In fact, overprediction in $[m] \setminus \set{J}^{=}(\bs{x}^\star)$ is admissible. For example, consider the LP $\min_{\bs{x} \geq 0}\{x_1 + x_2 ~ | ~ x_1 \geq 1, -x_1 \geq -2, x_2 \geq 1, -x_2 \geq -2\}$. The unique optimal solution is $\bs{x}^\star = (1, 1)$ and $\set{J}^{=}(\bs{x}^\star) = \{1, 3\}$. Suppose we make the prediction $\hat{b} = (0.5, -1.5, 0.5, -2.5)$ of the true right-hand side vector $\bs{b} = (1, -2, 1, -2)$. Then we overpredicted the second component, i.e., $\hat{\bs{b}}_2 > \bs{b}_2$, yet one can easily verify that $\bs{Ax}^\star \geq \hat{\bs{b}}$, thus $(\bs{x}^{\star}, \bs{y}^{\star}) \in \widehat{\set{S}}(\bs{\obs}; p)$.

Although underprediction of $\bs{b}$ guarantees $(\bs{x}^{\star}, \bs{y}^{\star})$ to reside in the predicted feasible region, enforcing $p$ to have such a property may lead to a loose estimate of $\hat{\bs{b}}$. Instead, our proposed optimistic and primal-DAL models \eqref{prob:optimisticTrain}, \eqref{prob:fixedPrimalTrain}, and \eqref{prob:fixedPrimalBoundTrain} incorporate a relaxed condition, $\bs{A} \bs{x}^{\star} \geq p(\bs{\obs})$, to ensure $(\bs{x}^{\star}, \bs{y}^{\star}) \in \widehat{\set{S}}(\bs{\obs}; p)$. Under this requirement, we identify the conditions for $(\bs{x}^{\star}, \bs{y}^{\star})$ to achieve optimality of the predicted problem. These results are stated in Proposition~\ref{prop:recoverySstar} and Corollary~\ref{cor:squeezing}.
\begin{proposition} \label{prop:recoverySstar}
    Consider an arbitrary quadruple $(\bs{\obs}, \bs{b}, \bs{x}^{\star}, \bs{y}^{\star})$ and let $p(\bs{\obs}) = \hat{\bs{b}}$. Suppose $(\bs{x}^{\star}, \bs{y}^{\star}) \in \widehat{\set{S}}(\bs{\obs}; p)$. We have $(\bs{x}^{\star}, \bs{y}^{\star}) \in \widehat{\set{S}}^{\star}(\bs{\obs}; p)$ if and only if $b_j = \hat{b}_j$ for all $j \in \set{J}^{+}(\bs{y}^{\star})$.
\end{proposition}

\begin{corollary}\label{cor:squeezing}
    Consider an arbitrary quadruple $(\bs{\obs}, \bs{b}, \bs{x}^{\star}, \bs{y}^{\star})$ and $p(\bs{\obs}) = \hat{\bs{b}}$. If $\bs{A} \bs{x}^{\star} \geq \hat{\bs{b}} \geq \bs{b}$ then $(\bs{x}^{\star}, \bs{y}^{\star}) \in \widehat{\set{S}}^{\star}(\bs{\obs}; p)$.
\end{corollary}

\subsection{Training Problems}
Hereafter, we focus on the class of linear predictors and present the training problems. 
Let us consider  $p(\bs{\xi}) = \bar{\bs{W}} \bs{\xi} + \bar{\bs{z}}$
where $\bar{\bs{W}}$ is a matrix of unknown weights and $\bar{\bs{z}}$ is the intercept of the model. This model can be equivalently written as $p(\bs{\xi}) = \bs{W \xi}$ where $\bs{W}$ is obtained by appending $\bar{\bs{z}}$ to $\bar{\bs{W}}$, i.e., $\bs{W} = [\bar{\bs{z}} \, | \, \bar{\bs{W}}]$, and the scalar $1$ is appended to the input $\bs{\xi}$. For notational convenience, we assume the intercept is implicitly handled by $\bs{W} \in \mathbb{R}^{m \times d}$. Additionally, in practice, $\bs{b}$ may consist of both unknown and determined components. In that case, it is desirable to only estimate the unknown components. While this reduces the dimension of prediction, we retain $p(\bs{\xi}) = \bs{W \xi}$ for simplicity, as this model accommodates such a partial prediction of $\bs{b}$ by some algebraic manipulations. 

When we aim to train a $\bs{W}$ using the dataset $\set{D}_N^\star$, most of the approaches proposed in \S~\ref{sec:diff_approaches} are high-dimensional problems. For example, in \eqref{prob:fixedPrimalTrain}, there are $(md + mN)$ variables in the problem while only $N$ observations are available. Motivated by the high-dimensional statistical learning literature, where the number of unknowns exceeds the number of available data points, we employ functions that are designed to promote sparsity, such as the $L_1$ norm proposed by \cite{tibshirani1996regression}. This leads us to the following training problem:
\begin{equation} \label{prob:fixedPrimalTrainLinear}
    \min_{\bs{W}, (\bs{y}_i)} \bigg \{
    F(\, \bs{W}, (\bs{y}_i)\,) ~\bigg \vert~ \bs{A}\bs{x}_{i}^{\star} \geq \bs{W\obs}_{i},~\bs{A}^\top \bs{y}_i \leq \bs{c},~ \bs{y}_i \geq \bs{0} \quad \forall i \in [N]\bigg \}, 
\end{equation}
where the objective function is defined as 
\begin{equation*}
    F( \, \bs{W}, (\bs{y}_i) \,) = \frac{1}{N} \sum_{i \in [N]} (\inner{\bs{c} , \bs{x}_{i}^{\star}} -  \inner{\bs{W\obs}_i, \bs{y}_i}) + \lambda \, r(\bs{W}) + \gamma \, \phi(\bs{W}).
\end{equation*}
Here, $r(\bullet)$ is a sparsity-inducing regularizer and $\phi(\bullet)$ measures the penalty of violating additional constraints, e.g., $\phi(\bs{W}) \coloneqq \sum_{i \in [N]}\sum_{j \in [m]}\max\{0, b_{ij} - \inner{\bs{w}_j, \bs{\obs}_i}\}$. 
We assume both $r$ and $\phi$ are convex functions, therefore, $F$ is a biconvex function, i.e., $F(\, \bullet,(\bs{y_i}) \,)$ is convex in $\bs{W}$ for a fixed $(\bs{y_i})$, and $F(\, \bs{W}, \bullet \,)$ is convex in $(\bs{y}_i)$ for a fixed $\bs{W}$. Lastly, both $\lambda$ and $\gamma$ are nonnegative weighting parameters. 

We note that the constraints of \eqref{prob:fixedPrimalTrainLinear} are separable in each variable. Since the dual feasible region $\set{Y} = \{\bs{y} \geq \bs{0} ~|~ \bs{A}^\top \bs{y} \leq \bs{c}\}$ is nonempty (this follows from an earlier assumption that for each $\bs{\obs} \in \rvset$, the optimal cost $v^\star$ of the C-LP \eqref{prob:downstream} is finite), we analyze the feasibility of the problem by investigating the first constraint. Proposition~\ref{prop:xi_sign_strict_feas} identifies conditions that guarantee a nonempty feasible set of \eqref{prob:fixedPrimalTrainLinear}. 
\begin{proposition}\label{prop:xi_sign_strict_feas}
Given $\bs{A}$ and $\set{D}_N^\star$, consider a set $\mathcal{W} \coloneqq \{ \bs{W} \in \mathbb{R}^{m \times d} ~\big \vert~ \bs{A}\bs{x}_{i}^{\star} \geq \bs{W\obs}_{i}, \, \forall i \in [N]\}$. The set $\mathcal{W}$ is nonempty if one of the following conditions hold: 
\begin{enumerate}
    \item[(i)] There exists $\tilde{k} \in [d]$ such that $\obs_{i \tilde{k}} > 0$ for all $i \in [N]$;
    \item[(ii)] There exists $\tilde{k} \in [d]$ such that $\obs_{i \tilde{k}} < 0$ for all $i \in [N]$;
    \item[(iii)] For every $k \in [d]$, either $\obs_{ik} \geq 0$ for all $i \in [N]$, or $\obs_{ik} \leq 0$ for all $i \in [N]$. Furthermore, $\bs{\obs}_i \neq \bs{0} \ \forall i \in [N]$.
\end{enumerate}
\end{proposition}

\subsubsection{Alternate Convex Search}
To solve \eqref{prob:fixedPrimalTrainLinear}, we apply a simple approach of iteratively solving for one variable while fixing the other. This approach, referred to as an alternate approach, was proposed by \cite{wendell1976minimization} to minimize a bivariate function subject to separable constraints. 
Algorithm~\ref{alg:acs} presents details of the alternate approach applied to our problem.
\begin{algorithm}
\caption{Alternate Convex Search \label{alg:acs}}
\begin{algorithmic}[1]
\State Parameters: $\lambda, \gamma > 0$;
\State Initialize $\bs{W}^{t}$, $(\bs{y}_i)^t$, and $t=0$;
\While{termination criteria are not satisfied}
\State Given $(\bs{y}_i)^t$, update
\begin{equation} \label{W_problem}
    \bs{W}^{t+1} \in \argmin \limits_{\bs{W}}
    \left\{ 
    \frac{1}{N} \sum \limits_{i \in [N]} (\inner{\bs{c} , \bs{x}_{i}^{\star}} -  \inner{\bs{W\obs}_i, \bs{y}_i^t}) + \lambda \, r(\bs{W}) + \gamma \, \phi(\bs{W})
    ~\bigg \vert~ \bs{A}\bs{x}_{i}^{\star} \geq \bs{W\obs}_{i}, \ \forall i \in [N]
    \right\};
\end{equation}
\State Given $\bs{W}^{t+1}$, update
\begin{equation} \label{y_problem}
    (\bs{y}_i)^{t+1} \in \argmin_{(\bs{y}_i)}
    \left\{ 
    \frac{1}{N} \sum \limits_{i \in [N]} (\inner{\bs{c} , \bs{x}_{i}^{\star}} -  \inner{\bs{W}^{t+1} \bs{\obs}_i, \bs{y}_i}) ~\bigg \vert~ \bs{A}^\top \bs{y}_i \leq \bs{c},~ \bs{y}_i \geq \bs{0} \ \forall i \in [N] \right\};
\end{equation}
\State $t \gets t + 1$;
\EndWhile
\State \Return $(\widehat{\bs{W}}, (\widehat{\bs{y}}_i)) = (\bs{W}^t, (\bs{y}_i)^t)$
\end{algorithmic}
\end{algorithm}

The convergence of the alternate approach has been shown in the literature.
\cite{wendell1976minimization} introduced a stationary solution suitable for bivariate minimization problems, called the partial optimal solution. The convergence property for the case of a biconvex program was formally stated in \cite{gorski2007biconvex}, identifying conditions under which the method yields a partial optimal solution. For a special case of a bilinear program,
\cite{Konno1976-ic} showed that a similar iterative scheme to the alternate approach, shown in \cite[Algorithm~1]{Konno1976-ic}, generates a Karush–Kuhn–Tucker (KKT) point, provided that the constraint sets are bounded. We state the convergence properties of Algorithm~\ref{alg:acs} in Theorem~\ref{thm:conv}.  

\begin{theorem} \label{thm:conv}
    Consider problem \eqref{prob:fixedPrimalTrainLinear}. Assume that $F$ is bounded below, and both $r$ and $\phi$ are convex functions. If the constraints $\mathcal{W} = \{ \, \bs{W} ~\big\vert~ \bs{A}\bs{x}_{i}^{\star} \geq \bs{W\obs}_{i}, \ \forall i \in [N] \, \}$ and $\bs{\mathcal{Y}} = \{ \, (\bs{y}_i) ~\big\vert~ \bs{A}^\top \bs{y}_i \leq \bs{c},~ \bs{y}_i \geq \bs{0}, \ \forall i \in [N] \, \}$ are bounded, 
    then the sequence $\{\bs{W}^{t}, (\bs{y_i})^{t} \}_{t=1}^\infty$ generated by Algorithm~\ref{alg:acs} satisfies
    \begin{enumerate}
        \item[(i)] The sequence $\{F(\bs{W}^{t}, (\bs{y_i})^{t}) \}_{t=1}^\infty$ is monotonically non-increasing; 
        \item[(ii)] \label{thm:part}
        Every accumulation point of  $\{F(\bs{W}^{t}, (\bs{y_i})^{t}) \}_{t=1}^\infty$ is a partial optimal solution, i.e., an accumulation point $(\bs{W}^*, (\bs{y}_i)^*)$ satisfies
        $$F(\bs{W}^*, (\bs{y}_i)^*) \leq F(\bs{W}, (\bs{y}_i)^*) \ \forall \, \bs{W} \in \mathcal{W} \quad \text{ and } \quad 
        F(\bs{W}^*, (\bs{y}_i)^*) \leq F(\bs{W}^*, (\bs{y}_i)) \ \forall \, (\bs{y}_i) \in \bs{\mathcal{Y}};$$
        \item[(iii)] Furthermore, if $r(\bullet)$ and $\phi(\bullet)$ are differentiable, a partial optimal solution of \eqref{prob:fixedPrimalTrainLinear} is equivalent to a KKT point of \eqref{prob:fixedPrimalTrainLinear}.
    \end{enumerate}
\end{theorem}

We note that an alternative approach to solving \eqref{prob:fixedPrimalTrainLinear} is to write the objective function as a difference-of-convex (DC) function and apply an algorithm designed for minimizing DC functions. A function $f(\bs{x})$ is called a DC function if there exist two convex functions $g(\bs{x})$ and $h(\bs{x})$ such that $f(\bs{x}) = g(\bs{x}) - h(\bs{x})$. For a DC function, identifying convex functions $g(\cdot)$ and $h(\cdot)$ may not always be possible; however, it turns out that by applying some algebraic work, we obtain a DC representation of \eqref{prob:fixedPrimalTrainLinear}. We present an explicit DC form of the objective in \eqref{prob:fixedPrimalTrainLinear} in Appendix \S\ref{appendix:dc}. Consequently, a numerical method minimizing a DC program, e.g., DC Algorithm in \cite{PhamHoiAn97, 6797753}, can be applied to compute a KKT point of \eqref{prob:fixedPrimalTrainLinear}, as shown in \cite{10.1007/978-3-319-06569-4_2, pang_razaviyayn_alvarado}.

\section{Numerical Experiments}\label{sec:numerical_experiments}
In this section, we report the results of numerical experiments evaluating the performance of the proposed DAL prediction models. For these experiments, we set the hypothesis class to linear prediction models i.e., $\set{P} = \{p~|~\exists \bs{W} \in \mathbb{R}^{m \times d} ~\text{s.t.}~p(\bs{\obs}) = \bs{W\obs}, \forall \bs{\obs} \in \rvset\}$. We conducted experiments on instances of synthetically generated C-LP problems and a network optimization problem. All experiments were conducted on a Windows 11 desktop with an Intel i7-10700 (16 threads) and 64 GB RAM.

For all instances of the two problems, we solve the optimistic-DAL problem \eqref{prob:optimisticTrain} and the primal-DAL problem \eqref{prob:fixedPrimalTrainLinear} with $p(\bs{\obs}) = \bs{W\obs}$. We solve the dual-DAL problem \eqref{prob:fixedDualTrain} with $p(\bs{\obs}) = \alpha\bs{W\obs} - \bs{b}$, where $\alpha \geq 0$ is a hyperparameter that we tune to avoid the trivial solution. We solve two variants of the primal-DAL problem both with a $L_1$ regularizer; the first does not include the constraint violation penalty obtained by setting $\gamma = 0$ in \eqref{prob:fixedPrimalTrainLinear} and the second is a penalized version with $\gamma \neq 0$ and $\phi(\bs{W}) \coloneqq \sum_{i \in [N]}\sum_{j \in [m]}\max\{0, b_{ij} - \inner{\bs{w}_j, \bs{\obs}_i}\}$. We benchmark the DAL problems against learning approaches that do not explicitly consider downstream decisions to predict the relationship between the context and the right-hand-side vectors. We utilize as benchmarks a linear regression (LR) model, a lasso regression model \citep{tibshirani1996regression}, and a random forests (RF) regression model \citep{breiman2001random} with 100 trees and $\lceil\frac{d}{3}\rceil$ features at each split.  
The exact form of the DAL problems, as well as the linear/lasso regression problems, is provided in Appendix \S\ref{app:details_and_additional_results_synthetic_data_experiment} for the synthetic experiment and in Appendix \S\ref{sec:network_optimization_details} for the network optimization experiment.
We solve the primal-DAL problem using the alternate convex search (Algorithm \ref{alg:acs}). This problem can be solved by using a commercial solver, and its DC representation can be tackled using the convex-concave procedure (Algorithm \ref{alg:ccp}), shown in Appendix \S\ref{appendix:dc}. We compare the alternative approaches whose details we present in 
Appendix \S\ref{app:details_and_additional_results_synthetic_data_experiment}. 
Our numerical comparison revealed that Algorithm \ref{alg:acs} is more efficient in solving this problem; hence, we use this solution method from here on out. We solve the optimistic and dual-DAL problems, which are both LPs, using Gurobi 12.0.1. We use the SciKit-Learn package \citep{pedregosa2011scikit} to implement regression-based prediction models.

Recall that the alternative DAL problems aim to minimally recover the true optimal solution as a feasible solution to the predicted problem and optimistically recover it as the optimal solution of the predicted problem. In light of this goal, we evaluate and compare the alternative training problems using the following metrics for a prediction outcome $p(\bs{\obs})$:
$$\text{Feasibility: }\chi\{\bs{Ax}^\star \geq p(\bs{\obs})\}  \qquad \text{{Duality} gap: } \inner{\bs{c}, \bs{x}^\star} - \inner{p(\bs{\obs}), \bs{y}^\star} $$
Here, $\chi\{\cdot\}$ is an indicator function that takes the value 1 if the input is true and 0 otherwise. If we only meet the minimal requirement, then we may not be able to recover the true optimal solution by optimizing the predicted problem. In fact, the optimal solution to the predicted problem ($\hat{\bs{x}}$) may not even be feasible for the true problem. In this case, we may project $\hat{\bs{x}}$ to the true feasible region $\set{X}(\bs{b})$ or the set of true optimal solutions $\set{X}^\star(\bs{b})$. We denote such a solution by $\tilde{\bs{x}}_i \in \arg\min_{\bs{x} \geq \bs{0}}\{||\bs{x} - \hat{\bs{x}}_i||_2^2 ~ | \bs{Ax} \geq \bs{b}_i\}$. We use the projection distance, denoted by $\Pi_\set{X} = ||\hat{\bs{x}}_i - \tilde{\bs{x}}_i||_2$, and the distance of the projected solution to the true optimal solution, denoted by $\Pi_{\set{X}^\star} = ||\tilde{\bs{x}}_i - \bs{x}_i^\star||_2$ to compare solutions from alternative DAL problems.


\subsection{Synthetic Data Experiment}\label{subsec:synthetic_data_experiments}
All instances of the synthetically generated C-LP problem \eqref{prob:downstream} have five decision variables ($n = 5$), seven constraints ($m = 7$), and three contextual features ($d = 3$). {In particular, we fix a cost vector $\bs{c} \in \RR^5$, a constraint matrix $\bs{A} \in \RR^{7 \times 5}$, and a ground truth matrix $\bs{W} \in \RR^{7 \times 3}$.}
We vary the training dataset size $N \in \{250, 500, 750, 1000\}$ and 
conduct $50$ replications, where each replication involves a C-LP instance with {an independently generated training sample $\{(\bs{\obs}_i, \bs{b}_i)\}$ of size $1000$ and validation sample $\{(\bs{\obs}_i^v, \bs{b}_i^v)\}$ of size 250}.
We perform hyperparameter tuning separately for each {training dataset size $N$}. For more details regarding the data-generation process and hyperparameter tuning, we refer the reader to Appendix \S\ref{app:details_and_additional_results_synthetic_data_experiment}.

In Table \ref{tab:percent_feas_true_soln_synthetic}, we report the results regarding the feasibility metric for the prediction models. We present these results as the percentage of the validation dataset where the true optimal solution $\bs{x}_i^\star$ resides in the feasible region of the predicted problem associated with the right-hand side generated using a specific training model.
\begin{table}[b!]
\footnotesize
\centering \renewcommand{\arraystretch}{1.2}
\begin{tabular}{cccccccc}
$N$ & Optimistic-DAL & Primal-DAL & Primal-DAL (w/ penalty) & Dual-DAL & {LR} & {Lasso} & {RF} \\
\hline
    250 & 93.09 & 96.77 & 94.50 & 44.00 & 24.92 & 26.58 & 24.13 \\
    500 & 96.56 & 98.00 & 97.24 & 46.19 & 24.82 & 26.40 & 24.14 \\
    750 & 97.68 & 98.66 & 98.10 & 46.74 & 25.11 & 26.52 & 24.33 \\
    1000 & 98.24 & 98.90 & 98.59  & 46.68 & 25.05 & 26.44 & 24.27 \\
\hline
\end{tabular}
\caption{Percentage of true solutions $\bs{x}_i^\star$ in the validation dataset which are in the predicted feasible regions.}
\label{tab:percent_feas_true_soln_synthetic}
\end{table}
The results indicate that the optimistic and primal-DAL problems recover a very high percentage $(> 90 \%)$ of the true solutions $\bs{x}_i^\star$ in their feasible regions. This is due to the fact that their models explicitly contain constraints $\bs{A}\bs{x}_i^\star \geq \bs{W\obs}_i$ for all $i \in [N]$. The inclusion of the penalty term {in the primal-DAL model} does not aid the solution quality with respect to the overall percentage, as evident from the third and fourth columns. With the feasibility percentage ranging only as high as {$46.74\%$}, the performance of the dual-DAL model deteriorates relative to the former two models, while still outperforming the regression models, which have very low true solution recovery percentages of around {$25\%$}. Moreover, as the size of the training dataset increases, the feasibility percentage also improves in most cases. From these results, we can conclude that including the effect of downstream decision-making in the learning process, as we do through constraints and objectives in the training optimization model, improves the feasibility metric. {In the companion Figure \ref{fig:true_soln_num_pred_constrs_satisfied}, we show the number of predicted constraints satisfied by the true solutions $\bs{x}_i^\star$ on one replication of the experiment with $N = 250$. The results show that a majority of the true solutions satisfy all seven constraints when the right-hand side values are predicted by the optimistic- and primal-DAL models (this corresponds to the $> 90\%$ shown in columns 2-4 in Table \ref{tab:percent_feas_true_soln_synthetic}). There are a few outliers that are infeasible for these models, and even then, at most two of the seven constraints are violated. On the other hand, when predictions are generated by the dual-DAL or regression models, most instances have at least one violated constraint, rendering $\bs{x}_i^\star$ infeasible (this elucidates the percentages reported in columns 5-8 for these models in Table \ref{tab:percent_feas_true_soln_synthetic}).}
\begin{figure}[t!]
    \centering
    \includegraphics[scale = 0.5]{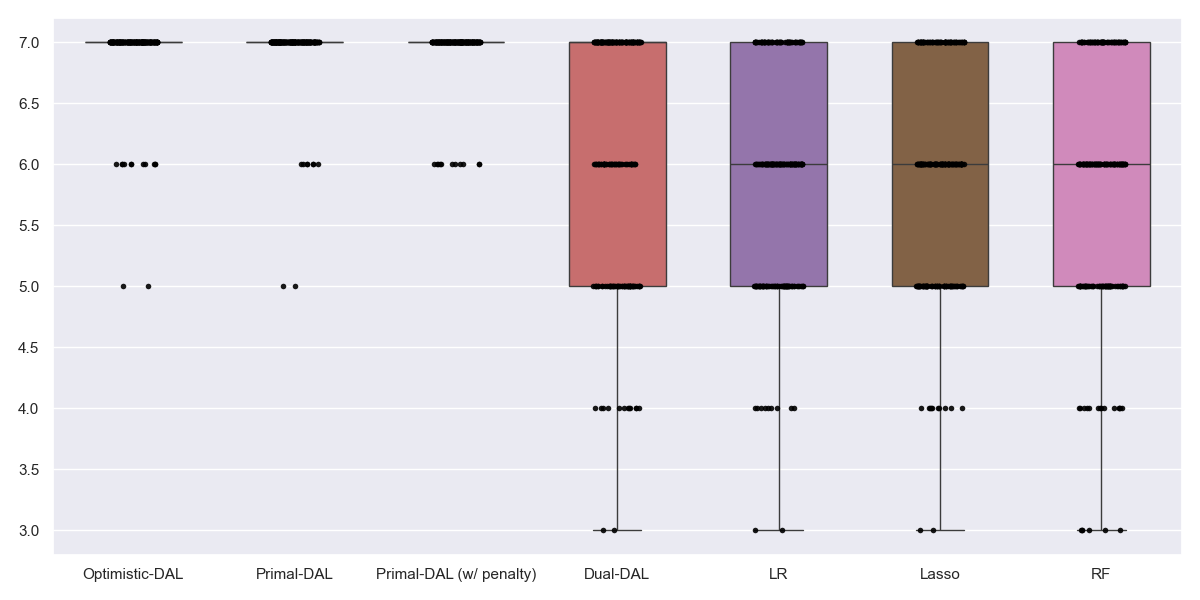}
    \caption{Number of predicted constraints satisfied by the true solutions $\bs{x}_i^\star$ {(result of one replication when training the models on $N = 250$ datapoints)}.}
    \label{fig:true_soln_num_pred_constrs_satisfied}
\end{figure}

Table \ref{tab:pred_duality_gap} shows the median {duality} gap of the predicted problem over all {validation} datapoints such that $\bs{x}_i^\star$ is in the predicted feasible region (the associated feasibility percentages from Table \ref{tab:percent_feas_true_soln_synthetic} are provided in parentheses).
\begin{table}[b!]
\footnotesize
\centering \renewcommand{\arraystretch}{1.2}
\resizebox{0.99\textwidth}{!}{\begin{tabular}{cccccccc}
$N$ & Optimistic-DAL & Primal-DAL & Primal-DAL (w/penalty) & Dual-DAL & {LR} & {Lasso} & {RF} \\
\hline
    250 & 5.88 (93.09\%) & 5.84 (96.77\%) & 5.74 (94.50\%) & 1.17 (44.00\%) & 0.47 (24.92\%) & 0.53 (26.58\%) & 0.70 (24.13\%) \\
    500 & 8.74 (96.56\%) & 8.62 (98.00\%) & 8.61 (97.24\%) & 1.20 (46.19\%) & 0.48 (24.82\%) & 0.53 (26.40\%) & 0.63 (24.14\%) \\
    750 & 10.31 (97.68\%) & 10.20 (98.66\%) & 10.25 (98.10\%) & 1.24 (46.74\%) & 0.48 (25.11\%) & 0.53 (26.52\%) & 0.60 (24.33\%) \\
    1000 & 11.69 (98.24\%) & 11.58 (98.90\%) & 11.57 (98.59\%) & 1.23 (46.68\%) & 0.48 (25.05\%) & 0.54 (26.44\%) & 0.60 (24.27\%) \\
\hline
\end{tabular}}
\caption{{Duality} gap of the true pair $(\bs{x}_i^\star, \bs{y}_i^\star)$ relative to the predicted problem.}
\label{tab:pred_duality_gap}
\end{table}
It is worthwhile to note that as the size of the training dataset ($N$) increases, the performance of the models that explicitly maintain feasibility across all data points (viz., optimistic- and primal-DAL in columns 2--4) deteriorates significantly with respect to the {duality}-gap metric. While the percentage of feasible points is lower in dual-DAL, among the datapoints where $\bs{x}_i^\star$ is feasible, the solution pair $(\bs{x}_i^\star, \bs{y}_i^\star)$ exhibits a lower {duality} gap for the predicted problem. {The percentage of feasible points is lowest for the regression models, and they exhibit the lowest duality gap among the datapoints where $\bs{x}_i^\star$ is feasible.}

\begin{table}[t!]
\small
\centering \renewcommand{\arraystretch}{1.2}
\resizebox{0.99\textwidth}{!}{\begin{tabular}{c|cc|cc|cc|cc|cc|cc|cc}
 & \multicolumn{2}{c|}{Optimistic-DAL} & \multicolumn{2}{c|}{Primal-DAL} & \multicolumn{2}{c|}{Primal-DAL (w/ penalty)} & \multicolumn{2}{c|}{Dual-DAL} & \multicolumn{2}{c|}{{LR}} & \multicolumn{2}{c|}{{Lasso}} & \multicolumn{2}{c}{{RF}} \\
$N$ & $\Pi_{\set{X}}$ & $\Pi_{\set{X}^\star}$ & $\Pi_{\set{X}}$ & $\Pi_{\set{X}^\star}$ & $\Pi_{\set{X}}$ & $\Pi_{\set{X}^\star}$ & $\Pi_{\set{X}}$ & $\Pi_{\set{X}^\star}$ & $\Pi_{\set{X}}$ & $\Pi_{\set{X}^\star}$ & $\Pi_{\set{X}}$ & $\Pi_{\set{X}^\star}$ & $\Pi_{\set{X}}$ & $\Pi_{\set{X}^\star}$ \\
\hline
250 & 1.81 & 1.47 & 1.86 & 1.50 & 1.80 & 1.50 & 0.28 & 0.28 & 0.08 & 0.21 & 0.08 & 0.21 & 0.10 & 0.27 \\
500 & 2.96 & 2.23 & 2.98 & 2.23 & 2.96 & 2.23 & 0.29 & 0.29 & 0.08 & 0.21 & 0.08 & 0.21 & 0.09 & 0.25 \\
750 & 3.59 & 2.78 & 3.61 & 2.74 & 3.59 & 2.79 & 0.31 & 0.29 & 0.08 & 0.21 & 0.08 & 0.21 & 0.09 & 0.24 \\
1000 & 4.25 & 3.36 & 4.25 & 3.37 & 4.25 & 3.41 & 0.31 & 0.29 & 0.08 & 0.21 & 0.08 & 0.21 & 0.09 & 0.24 \\
\hline
\end{tabular}}
\caption{Projection distances $\Pi_{\set{X}}$ and $\Pi_{\set{X}^\star}$.}
\label{tab:synthetic_proj}
\end{table}

While the DAL models reliably recover the true optimal solution in the predicted feasible region, as indicated by the results in Table \ref{tab:percent_feas_true_soln_synthetic}, we do not have a suitable approach to identify the true optimal solution $\bs{x}_i^\star$. In our final experiment, we investigate using the optimal solution to the predicted problem, $\hat{\bs{x}}$, as a proxy for the true optimal solution.
Table \ref{tab:synthetic_proj} displays the median projection distances $\Pi_{\set{X}}$ and $\Pi_{\set{X}^\star}$.
{For this experiment, each model was always able to generate a feasible and bounded C-LP \eqref{prob:downstream} (however, the solution $\hat{\bs{x}}$ generated by our models is seldom feasible to the true problem -- see Table \ref{tab:percent_feas_pred_soln_synthetic} in Appendix \S\ref{app:details_and_additional_results_synthetic_data_experiment}).}
{Among our proposed models,} the predicted problems associated with dual-DAL generate solutions that are closest to the true feasible region and to optimal solutions. It is also worth noticing that as the size of the training set $N$ increases, the predicted optimal solution obtained either from the optimistic- or primal-DAL models lies further away from the true feasible region. {On the other hand, the projection metrics remain relatively unaffected for the Dual-DAL and the regression models as $N$ increases. In fact, we observe no difference in the median projection distances for the linear regression and lasso regression models measured at two significant figures past the decimal. Upon further investigation, we saw that for the linear and lasso regression models, the model coefficients changed as $N$ increased, however, the change lead to minimal differences in the values of the predictions $\hat{\bs{b}}$ and in the values of the resulting downstream solutions $\hat{\bs{x}}$.}


\subsection{Network Optimization Problem}
\begin{table}[b!]
\footnotesize
\centering \renewcommand{\arraystretch}{1.2}
\resizebox{0.99\textwidth}{!}{\begin{tabular}{cccccccc} 
Metric & Optimistic-DAL & Primal-DAL & Primal-DAL (w/ penalty) & Dual-DAL & {LR} & {Lasso} & {RF} \\ 
\hline
{Duality} Gap & 57792.02 & 57808.67 & 57750.25 & 53324.8 & 49585.67 & 49585.29 & 58185.82 \\ \\
\hline
$\Pi_{\set{X}}$ & 6959.42 & 6967.04 & 6952.95 & 4802.99 & 644.88 & 644.89 & 584.26 \\
\hline
$\Pi_{\set{X}^\star}$ & 21728.62 & 21735.02 & 21719.97 & 24466.48 & 9976.28 & 9975.50 & 10582.88 \\
\hline
\end{tabular}}
\caption{Performance of models on the network optimization problem \eqref{prob:network} {(the duality gap is computed for validation datapoints such that $\bs{x}_i^\star$ is in the corresponding predicted feasible region, and the projection distances $\Pi_{\set{X}}$ and $\Pi_{\set{X}^\star}$ are computed over all validation datapoints)}}
\label{tab:network_optimization_summary}
\end{table}
We consider a minimum-cost network flow problem involving a set of source, transhipment, and destination nodes. In addition to the shipment costs, to ensure that the optimization problem remains feasible with variations in parameters, we introduce a penalty cost for unmet demand at the destination nodes. In this problem, demand is uncertain and depends on a context vector comprising local average daily temperature, day of the week, and month. The optimization problem contains 75 decision variables and 24 constraints. Of the constraints, five have right-hand side components that correspond to the contextual vector. We refer the reader to Appendix \S\ref{sec:network_optimization_details} for a detailed presentation of the optimization model, contextual features, and hyperparameter tuning.

For our experiments on the network optimization problem, we utilize a real-world dataset to draw independent samples for each replication. We use approximately 75\% of the sampled data for training and the remaining 25\% for validation. Our experiments reveal that the predicted feasible region obtained using the optimistic-DAL, primal-DAL, and penalized primal-DAL models contains the true optimal solution in $84.35\%$, $84.25\%$, and {$84.22\%$} of the validation instances, respectively. Compared to the synthetic problem, the feasibility metric was much lower at {$1.01\%$} for dual-DAL. Finally, the linear, lasso, and random-forest regression models have feasibility metric values of $13.69\%$, $13.69\%$, and $11.72\%$, respectively. The median number of predicted constraints that a true solution $\bs{x}_i^\star$ satisfies is five (out of the possible five) for the optimistic and primal-DAL models, and only two out of five for all other models. These results provide further evidence of the value of decision-aware prediction models. 

Table \ref{tab:network_optimization_summary} shows the results pertaining to the {duality} gap and projection distances for the network optimization problem. {Note that the duality gap metric is computed only over those validation datapoints where $\bs{x}_i^\star$ is in the predicted feasible region (see the above feasibility percentages), as in the synthetic experiment. On the other hand, the projection distances $\Pi_{\set{X}}$ and $\Pi_{\set{X}^\star}$ are computed for each validation datapoint where the corresponding prediction model generates a feasible and bounded C-LP \eqref{prob:downstream} (which happened to be all $100\%$ of the datapoints for every prediction model).} As in the synthetic problem, the performance of the optimistic- and primal-DAL models {as well as the regression models} is similar. However, unlike the synthetic problem, the dual-DAL model performs relatively worse on the {duality} gap and projection distance metrics, as seen in the {fifth} column of the table. This behavior, along with the low value of the feasibility metric, is attributed to setting the hyperparameter $\alpha = 2$ rather than tuning it. 


\section{Conclusions}
In this paper, we propose alternative formulations for training a model to predict the right-hand side of an LP using a correlated contextual vector. Using observed primal and dual optimal solutions of the LP, our formulations aim to increase the feasibility of the predicted problem with respect to the true optimal solution while minimizing its duality gap. We analyze properties of the training problems, identify conditions under which the resulting prediction model recovers the mentioned feasibility and optimality, and present suitable solution methods to solve each problem. The proposed methods are validated through numerical experiments on synthetic and network optimization problems. The results show that the prediction models trained using the proposed formulation achieve much higher feasibility, compared to standard regression approaches, for the unseen (validation) dataset. The results also indicate that as the number of training data points increases, the feasibility of the model enhances at the cost of the {duality} gap.

\subsubsection*{Notes}

This second submission has some notable changes from the first submission. In particular, in the synthetic data experiment in \S\ref{subsec:synthetic_data_experiments}, we now fix the cost vector $\bs{c}$, the constraint matrix $\bs{A}$ and ground truth model $\bs{W}$ so that they are the same for every replication. Additionally, we fixed a bug that was present in the computation of the duality gap metric (formerly called the optimality gap metric) in Table \ref{tab:network_optimization_summary} for the network optimization problem.

\subsubsection*{Acceptance Status}

This paper has been accepted in the 2026 INFORMS Optimization Society Refereed Proceedings.

\subsubsection*{Funding Acknowledgment}

All three authors’ work was supported by the National Science Foundation under grant CCF-2423246. The first and the third author also acknowledge the Office of Naval Research Grant \#N00014-22-1-2603 for partially supporting this work.

\bibliographystyle{plain}
\bibliography{master}

@article{bengio1997using,
  title={Using a financial training criterion rather than a prediction criterion},
  author={Bengio, Yoshua},
  journal={International journal of neural systems},
  volume={8},
  number={04},
  pages={433--443},
  year={1997},
  publisher={World Scientific}
}

@article{elmachtoub2022smart,
  title={Smart “predict, then optimize”},
  author={Elmachtoub, Adam N and Grigas, Paul},
  journal={Management Science},
  volume={68},
  number={1},
  pages={9--26},
  year={2022},
  publisher={INFORMS}
}

@inproceedings{mandi2020smart,
  title={Smart predict-and-optimize for hard combinatorial optimization problems},
  author={Mandi, Jayanta and Stuckey, Peter J and Guns, Tias and others},
  booktitle={Proceedings of the AAAI conference on artificial intelligence},
  volume={34},
  number={02},
  pages={1603--1610},
  year={2020}
}

@inproceedings{demirovic2020dynamic,
  title={Dynamic Programming for Predict+ Optimise.},
  author={Demirovic, Emir and Stuckey, Peter J and Guns, Tias and Bailey, James and Leckie, Christopher and Ramamohanarao, Kotagiri and Chan, Jeffrey and others},
  booktitle={AAAI},
  pages={1444--1451},
  year={2020}
}

@article{estes2023smart,
  title={Smart predict-then-optimize for two-stage linear programs with side information},
  author={Estes, Alexander S and Richard, Jean-Philippe P},
  journal={INFORMS Journal on Optimization},
  volume={5},
  number={3},
  pages={295--320},
  year={2023},
  publisher={INFORMS}
}

@article{hu2023two,
  title={Two-stage predict+ optimize for MILPs with unknown parameters in constraints},
  author={Hu, Xinyi and Lee, Jasper and Lee, Jimmy},
  journal={Advances in Neural Information Processing Systems},
  volume={36},
  pages={14247--14272},
  year={2023}
}

@article{sadana2024survey,
  title={A survey of contextual optimization methods for decision-making under uncertainty},
  author={Sadana, Utsav and Chenreddy, Abhilash and Delage, Erick and Forel, Alexandre and Frejinger, Emma and Vidal, Thibaut},
  journal={European Journal of Operational Research},
  year={2024},
  publisher={Elsevier}
}

@article{tibshirani1996regression,
  title={Regression shrinkage and selection via the lasso},
  author={Tibshirani, Robert},
  journal={Journal of the Royal Statistical Society Series B: Statistical Methodology},
  volume={58},
  number={1},
  pages={267--288},
  year={1996},
  publisher={Oxford University Press}
}

@article{breiman2001random,
  title={Random forests},
  author={Breiman, Leo},
  journal={Machine learning},
  volume={45},
  pages={5--32},
  year={2001},
  publisher={Springer}
}

@article{PhamHoiAn97,
  year = {1997},
  volume = {22},
  number = {1},
  pages = {289--355},
  author = {Tao {Pham Dinh} and Hoai An {Le Thi}},
  title = {Convex analysis approach to {D.C.} programming: Theory, algorithms and applications},
  journal = {ACTA Mathematica Vietnamica}
}

@ARTICLE{6797753,
  author={Sriperumbudur, Bharath K. and Lanckriet, Gert R. G.},
  journal={Neural Computation}, 
  title={A Proof of Convergence of the Concave-Convex Procedure Using Zangwill's Theory}, 
  year={2012},
  volume={24},
  number={6},
  pages={1391-1407},
  doi={10.1162/NECO_a_00283}}

@InProceedings{10.1007/978-3-319-06569-4_2,
author="Le Thi, Hoai An
and Huynh, Van Ngai
and Dinh, Tao Pham",
editor="van Do, Tien
and Thi, Hoai An Le
and Nguyen, Ngoc Thanh",
title="DC Programming and DCA for General DC Programs",
booktitle="Advanced Computational Methods for Knowledge Engineering",
year="2014",
publisher="Springer International Publishing",
address="Cham",
pages="15--35",
isbn="978-3-319-06569-4"
}

@article{pang_razaviyayn_alvarado,
author = {Pang, Jong-Shi and Razaviyayn, Meisam and Alvarado, Alberth},
title = {Computing B-Stationary Points of Nonsmooth DC Programs},
journal = {Mathematics of Operations Research},
volume = {42},
number = {1},
pages = {95-118},
year = {2017},
doi = {10.1287/moor.2016.0795}
}

@article{yuille2003concave,
  title={The concave-convex procedure},
  author={Yuille, Alan L and Rangarajan, Anand},
  journal={Neural computation},
  volume={15},
  number={4},
  pages={915--936},
  year={2003},
  publisher={MIT Press}
}

@article{lipp2016variations,
  title={Variations and extension of the convex--concave procedure},
  author={Lipp, Thomas and Boyd, Stephen},
  journal={Optimization and Engineering},
  volume={17},
  pages={263--287},
  year={2016},
  publisher={Springer}
}

@article{gorski2007biconvex,
  title={Biconvex sets and optimization with biconvex functions: a survey and extensions},
  author={Gorski, Jochen and Pfeuffer, Frank and Klamroth, Kathrin},
  journal={Mathematical methods of operations research},
  volume={66},
  number={3},
  pages={373--407},
  year={2007},
  publisher={Springer}
}

@article{wendell1976minimization,
  title={Minimization of a non-separable objective function subject to disjoint constraints},
  author={Wendell, Richard E and Hurter Jr, Arthur P},
  journal={Operations Research},
  volume={24},
  number={4},
  pages={643--657},
  year={1976},
  publisher={INFORMS}
}

@ARTICLE{Konno1976-ic,
  title     = "A cutting plane algorithm for solving bilinear programs",
  author    = "Konno, Hiroshi",
  journal   = "Math. Program.",
  publisher = "Springer Science and Business Media LLC",
  volume    =  11,
  number    =  1,
  pages     = "14--27",
  month     =  dec,
  year      =  1976
}

@book{boyd2004convex,
  title={Convex optimization},
  author={Boyd, Stephen and Vandenberghe, Lieven},
  year={2004},
  publisher={Cambridge university press}
}

@book{bishop2006pattern,
  title={Pattern recognition and machine learning},
  author={Bishop, Christopher M and Nasrabadi, Nasser M},
  volume={4},
  number={4},
  year={2006},
  publisher={Springer}
}

@misc{erickson_county_city_driving_2014,
  author = {Erickson, Jeff},
  title = {County/City Driving Distance Dataset: Driving distances for each county centroid to the nearest large city in the contiguous United States},
  year = {2014},
  url = {https://github.com/jefferickson/county-city-driving-dist},
  note = {Accessed: Novemver 4, 2025},
}

@article{pedregosa2011scikit,
  title={Scikit-learn: Machine learning in Python},
  author={Pedregosa, Fabian and Varoquaux, Ga{\"e}l and Gramfort, Alexandre and Michel, Vincent and Thirion, Bertrand and Grisel, Olivier and Blondel, Mathieu and Prettenhofer, Peter and Weiss, Ron and Dubourg, Vincent and others},
  journal={the Journal of machine Learning research},
  volume={12},
  pages={2825--2830},
  year={2011},
  publisher={JMLR. org}
}




\appendix

\titleformat{\section}
  {\normalfont\Large\bfseries}
  {Appendix~\Alph{section}:}
  {1em}
  {}
\section{Proofs of the Results} \label{sec:appendix_proof}
This section includes the proofs of all the results that appear in the paper. 
\proof{Proof of Proposition~\ref{prop:recoveryS}:} 
$(i)$ Observe that $\bs{b} \geq  \hat{\bs{b}}$ implies $\{ \bs{x} \geq \bs{0} ~ | ~ \bs{A} \bs{x} \geq \bs{b} \} \subseteq \{ \bs{x} \geq \bs{0} ~ | ~ \bs{A} \bs{x} \geq \hat{\bs{b}} \}$. Therefore, we must have $(\bs{x}^{\star}, \bs{y}^{\star}) \in \widehat{\set{S}}(\bs{\obs}; p)$. $(ii)$ Since $j' \in \set{J}^{=}(\bs{x}^{\star})$, we have $\inner{\bs{a}_{j'}, \bs{x}^{\star}} = b_{j'} < \hat{b}_{j'}$. This completes the proof. 
\endproof

\proof{Proof of Proposition~\ref{prop:recoverySstar}:} $(\implies)$ Since $\widehat{S}^\star(\bs{\obs}; p) \subseteq \widehat{S}(\bs{\obs}; p)$ and by part $(ii)$ of Proposition \eqref{prop:recoveryS}, we have $\hat{b}_j \leq b_j$ for all $j \in \set{J}^+(\bs{y}^\star) \subseteq \set{J}^=(\bs{x}^\star)$. Moreover,
\begin{equation*}
    \inner{\hat{\bs{b}}, \bs{y}^{\star}} = \inner{\bs{c}, \bs{x}^{\star}} = \inner{\bs{b}, \bs{y}^{\star}},
\end{equation*}
where the last equality follows by strong duality of $(\bs{x}^{\star}, \bs{y}^{\star})$. Hence, $\sum_{j \in \set{J}^+(\bs{y}^\star)}(\hat{b}_j - b_j)y_j^\star = 0$, implying that $\hat{b}_j = b_j$ for all $j \in \set{J}^+(\bs{y}^\star)$. $(\impliedby)$ Applying the definition of $\set{J}^{+}(\bs{y}^{\star})$ to the condition of the proposition yields
\begin{equation*}
    \inner{\hat{\bs{b}}, \bs{y}^{\star}}
    = \displaystyle{\sum \limits_{j \in \set{J}^{+}(\bs{y}^{\star})}} \, \hat{b}_j \, y^{\star}_j 
    = \displaystyle{\sum \limits_{j \in \set{J}^{+}(\bs{y}^{\star})}} \, b_j \, y^{\star}_j
    =\inner{\bs{b}, \bs{y}^{\star}} 
    = \inner{\bs{c}, \bs{x}^{\star}},
\end{equation*}
where the last equality is followed by the strong duality of $(\bs{x}^{\star}, \bs{y}^{\star})$. Therefore, $(\bs{x}^{\star}, \bs{y}^{\star}) \in \widehat{\set{S}}^{\star}(\bs{\obs}; p)$.
\endproof

\proof{Proof of Corollary~\ref{cor:squeezing}:} By definition, $\inner{\bs{a}_{j}, \bs{x}^{\star}} = b_{j}$ for all $j \in \set{J}^{=}(\bs{x}^{\star})$, which also holds for all $j \in \set{J}^{+}(\bs{y}^{\star})$. Applying the condition of corollary yields $b_j = \hat{b}_j$ for any $j \in \set{J}^{+}(\bs{y}^{\star})$. By Proposition~\ref{prop:recoverySstar}, we have $(\bs{x}^{\star}, \bs{y}^{\star}) \in \widehat{\set{S}}^{\star}(\bs{\obs}; p)$.
\endproof

\proof{Proof of Proposition~\ref{prop:xi_sign_strict_feas}:} Consider an arbitrary $j \in [m]$. 
Let us denote the $j$-th component of $\bs{A} \bs{x}_i^\star$ as $\theta_{ij}$.  The corresponding $j$-th constraint of $\bs{A}\bs{x}_{i}^{\star} \geq \bs{W\obs}_{i}, \, \forall i \in [N],$ can be viewed as an intersection of hyperplanes $\cap_{i \in [N]} \{ \bs{w} \in \mathbb{R}^d ~\big \vert~ \theta_{ij} \geq \langle \bs{w}, \bs{\obs}_i \rangle \}$. 

To show $(i)$, we construct a feasible $\tilde{\bs{w}} \in \mathbb{R}^d$ by setting 
\begin{equation}\label{prop:feasibility}
\tilde{w}_{k} = 
\begin{cases}
    \min \limits_{i' \in [N], \, j' \in [m]} \, \left\{ \, \displaystyle{\frac{\theta_{i' j'}}{\obs_{i' k}}} \, \right\}, & \text{ if } k = \tilde{k}  \\
    \hspace{1.2pc} 0, & \text{ otherwise.}  
\end{cases}
\end{equation}
With strict positivity of $\obs_{i \mkern1mu \tilde{k}}$, the above then yields, 
\begin{align*}
    \langle \tilde{\bs{w}}, \bs{\obs}_i \rangle
    = 
    \min \limits_{i' \in [N], \, j' \in [m]} \left\{ \, \frac{\theta_{i' j'}}{\obs_{i' \mkern1mu \tilde{k}}} \, \right\} \, \obs_{i \mkern1mu \tilde{k}} 
    \, \leq \,
    \frac{\min \limits_{i' \in [N], j' \in [m]} \{ \, \theta_{i' j'} \, \} }{\obs_{i \mkern1mu \tilde{k}}} \, \obs_{i \mkern1mu \tilde{k}}
    \, \leq \,
    \theta_{ij} \text{ for any } i \in [N].
\end{align*}
By applying \eqref{prop:feasibility} to every row of $\bs{W}$, we show that $\mathcal{W}$ is nonempty. The proof for part $(ii)$ is identical except that we assign $\max \limits_{i' \in [N], \, j' \in [m]} \, \left\{ \, \displaystyle{\frac{\theta_{i' j'}}{\obs_{i' k}}} \, \right\}$ to $\tilde{w}_k$ if $k = \tilde{k}$, and $0$ otherwise. 

To show $(iii)$, define $\set{K}^+ \coloneqq \{ k \in [d] ~\vert ~ \obs_{ik} \geq 0, \ \forall i \in [N] \}$ and $\set{K}^- \coloneqq \{ k \in [d] ~\vert ~ \obs_{ik} \leq 0, \ \forall i \in [N] \}$ such that $\set{K}^+ \cap \set{K}^- = \emptyset$. Let $\sigma_i \coloneqq \Big( \sum \limits_{k \in \set{K}^+} \obs_{ik} - \sum \limits_{k \in \set{K}^-} \obs_{ik} \Big) > 0$. Construct $\tilde{w}$ such that, 
\begin{equation*}
\tilde{w}_{k} = 
\begin{cases}
    \min \limits_{i' \in [N], \, j' \in [m]} \, \left\{ \, \displaystyle{\frac{\theta_{i' j'}}{\sigma_{i'}}} \, \right\}, & \text{ if } k \in \set{K}^+  \\
    - \min \limits_{i' \in [N], \, j' \in [m]} \, \left\{ \, \displaystyle{\frac{\theta_{i' j'}}{\sigma_{i'}}} \, \right\}, & \text{ if } k \in \set{K}^-. 
\end{cases}
\end{equation*}
Consequently, we have
\begin{align*}
    \langle \tilde{\bs{w}}, \bs{\obs}_i \rangle
    &=
    \sum \limits_{k \in \set{K}^+} \tilde{w}_k \, \obs_{ik} + \sum \limits_{k \in \set{K}^-} \tilde{w}_k \, \obs_{ik} \\
    &=
    \sum \limits_{k \in \set{K}^+} \min \limits_{i' \in [N], \, j' \in [m]} \, \left\{ \, \displaystyle{\frac{\theta_{i' j'}}{\sigma_{i'}}} \, \right\} \, \obs_{ik} + \sum \limits_{k \in \set{K}^-} \min \limits_{i' \in [N], \, j' \in [m]} \, \left\{ \, \displaystyle{\frac{\theta_{i' j'}}{\sigma_{i'}}} \, \right\} \, | \, \obs_{ik} \, | \\
    &\leq 
    \min \limits_{i' \in [N], \, j' \in [m]} \, \left\{ \, \theta_{i' j'} \, \right\} \,
    \frac{1}{\sigma_{i}} \Big( \, \underbrace{\sum \limits_{k \in \set{K}^+} \obs_{ik} + \sum \limits_{k \in \set{K}^-} | \, \obs_{ik} \, | }_{= \, \sigma_i} \Big) \\
    &\leq
    \theta_{ij} \text{, for any $i \in [N]$.}
\end{align*}
This concludes the proof.
\endproof


\proof{Proof of Theorem \ref{thm:conv}:} 
We will prove Theorem \ref{thm:conv} using a biconvex program with separable constraints:
    \begin{equation} \label{prob:biconvex_separable}
        \min \limits_{\bs{x}, \bs{y}} \ \left\{ f(\bs{x}, \bs{y}) ~\bigg \vert~ \bs{x} \in X, \, \bs{y} \in Y \right\},
    \end{equation}
where $f:X \times Y \rightarrow \mathbb{R}$ is a biconvex function. We assume $X \coloneqq \{ \bs{x} ~ | ~ g_i(\bs{x}) \leq 0, ~ \forall i \in \set{I}\}$ for some index set $\set{I}$, where $g_i$ are differentiable and that $Y \coloneqq \{\bs{y} ~ | ~ h_j(\bs{y}) \leq 0, ~ \forall j \in \set{J}\}$ for some index set $\set{J}$, where $h_j$ are differentiable. A partial optimal solution of the problem is defined below.
\begin{definition}
    A point $(\bs{x}^*, \bs{y}^*)$ is a partial optimal solution of \eqref{prob:biconvex_separable} if it satisfies
        \begin{equation*}
        f(\bs{x}^*, \bs{y}^*) \leq f(\bs{x}, \bs{y}^*) \ \forall x \in X \quad \text{ and } \quad f(\bs{x}^*, \bs{y}^*) \leq f(\bs{x}^*, \bs{y}) \ \forall y \in Y.
    \end{equation*}
\end{definition}
Suppose we apply Alternate Convex Search (ACS) in \cite{gorski2007biconvex} to solve the problem. The steps of ACS are described below. Given $t=0$ and an initial $(\bs{x}^t, \bs{y}^t)$, sequentially update
    \begin{align} 
        \bs{x}^{t+1} \in \argmin \limits_{\bs{x}} \left\{ f (\bs{x}, \bs{y}^t) ~\bigg \vert~ \bs{x} \in X \right\}, \quad \label{acs:x}  \\ 
        \bs{y}^{t+1} \in \argmin \limits_{\bs{y}} \left\{ f (\bs{x}^{t+1}, \bs{y}) ~\bigg \vert~ \bs{y} \in Y \right\}, \label{acs:y}
    \end{align}
and $t \leftarrow t+1$ until the stopping criteria are satisfied. If $f$ is a biconvex function that is bounded below, and $X$ and $Y$ are compact sets, then 
\begin{enumerate}
    \item[($i$)] The sequence $\{ f(\bs{x}^t, \bs{y}^t) \}_{t=1}^\infty$ is monotonically non-increasing;
    \item[($ii$)] Every accumulation point of  $\{ (\bs{x}^t, \bs{y}^t) \}_{t=1}^\infty$ is a partial optimal solution. 
    \item[($iii$)] Furthermore, if $f$ is differentiable, a partial optimal solution of \eqref{prob:biconvex_separable} is a KKT point of \eqref{prob:biconvex_separable}.
\end{enumerate}
\begin{proof}{Proof of Theorem \ref{thm:conv}:}

($i$) It is not difficult to see that $f(\bs{x}^t, \bs{y}^t) \geq f(\bs{x}^{t+1}, \bs{y}^{t+1})$ for all $t$ by the optimality of \eqref{acs:x} and \eqref{acs:y}. 

($ii$) By Bolzano-Weierstrass theorem, $\{(\bs{x}^t, \bs{y}^t)\}_{t=1}^\infty$ has a convergent subsequence, denoted by $(\bs{x}^{t_j}, \bs{y}^{t_j}) \rightarrow (\bs{x}^*, \bs{y}^*)$ as $t_j \rightarrow \infty$. For any $t_j$, we have $f(\bs{x}^{{t_j} + 1}, \bs{y}^{{t_j} + 1}) \leq f(\bs{x}, \bs{y}^{t_j})$ for all $x \in X$ and $f(\bs{x}^{{t_j} + 1}, \bs{y}^{{t_j} + 1}) \leq f(\bs{x}^{t_j}, \bs{y})$ for all $y \in Y$ by \eqref{acs:x} and \eqref{acs:y}. By part $(i)$ and taking the limit, the former inequality yields $f(\bs{x}^*, \bs{y}^*) = \lim \limits_{t_j \rightarrow \infty} f(\bs{x}^{{t_j}}, \bs{y}^{{t_j}}) = \lim \limits_{t_j \rightarrow \infty} f(\bs{x}^{{t_j} + 1}, \bs{y}^{{t_j} + 1}) \leq f(\bs{x}, \bs{y}^*)$ for all $x \in X$. Likewise, $f(\bs{x}^*, \bs{y}^*) \leq f(\bs{x}^*, \bs{y})$ for all $y \in Y$, which shows $(\bs{x}^*, \bs{y}^*)$ is a partial optimal solution.

($iii$) Let $(\bs{x}^*, \bs{y}^*)$ be a partial optimum of \eqref{prob:biconvex_separable}, i.e., $f(\bs{x}^*, \bs{y}^*) = \min_{\bs{x}}\{f(\bs{x}, \bs{y}^*) ~ | ~ g_i(\bs{x}) \leq 0, ~ \forall i \in \set{I}\}$ and $f(\bs{x}^*, \bs{y}^*) = \min_{\bs{y}}\{f(\bs{x}^*, \bs{y}) ~ | ~ h_{j}(\bs{y}) \leq 0, ~ \forall j \in \set{J}\}$. The point $\bs{x}^*$ is a global minimizer for the former optimization problem, hence it is a KKT point for this problem \citep{boyd2004convex}. Thus, there is some $\bs{\lambda}^* \in \RR_+^{|\set{I}|}$ such that $g_i(\bs{x}^*) \leq 0$ and $\lambda_i^* g_i(\bs{x}^*) = 0$ for all $i \in \set{I}$, and $\nabla_{\bs{x}}f(\bs{x}^*, \bs{y}^*) + \sum_{i \in \set{I}}\lambda_i^*\nabla_{\bs{x}}g_i(\bs{x}^*) = \bs{0}$. Using the same logic, we have that there is some $\bs{\mu}^* \in \RR_+^{|\set{J}|}$ such that $h_{j}(\bs{y}^*) \leq 0$ and $\mu_{j}^* h_{j}(\bs{y}^*) = 0$ for all $j \in \set{J}$, and $\nabla_{\bs{y}}f(\bs{x}^*, \bs{y}^*) + \sum_{j \in \set{J}}\mu_{j}^*\nabla_{\bs{y}}h_{j}(\bs{y}^*) = \bs{0}$. The union of the primal feasibility conditions for the $\bs{x}$ and $\bs{y}$-subproblems is the same as the primal feasibility KKT condition for a point $(\bs{x}^*, \bs{y}^*)$ in the original separable biconvex program \eqref{prob:biconvex_separable} (this is also true for the dual feasibility as well as the complementary slackness conditions). That is, the primal/dual feasibility conditions as well as the complementary slackness condition hold for the point $(\bs{x}^*, \bs{y}^*)$ in \eqref{prob:biconvex_separable} with dual vectors $(\bs{\lambda}^*, \bs{\mu}^*)$. Since
\begin{equation*}
    \nabla_{\bs{x}}f(\bs{x}^*, \bs{y}^*) + \sum_{i \in \set{I}}\lambda_i^*\nabla_{\bs{x}}g_i(\bs{x}^*) = 0 \qquad \text{and} \qquad \nabla_{\bs{y}}f(\bs{x}^*, \bs{y}^*) + \sum_{j \in \set{J}}\mu_{j}^*\nabla_{\bs{y}}h_{j}(\bs{y}^*) = 0.
\end{equation*}
then
\begin{equation*}
    \nabla f(\bs{x}^*, \bs{y}^*) + \sum_{i \in \set{I}}\lambda_i^*\nabla g_i(\bs{x}^*) + \sum_{j \in \set{J}}\mu_i^*\nabla h_{j}(\bs{y}^*) = 0.
\end{equation*}
This concludes the proof.
\end{proof}

\section{A Difference-of-convex Representation of \eqref{prob:fixedPrimalTrainLinear}} \label{appendix:dc}
By applying some algebraic techniques, 
we identify a DC representation of the problem \eqref{prob:fixedPrimalTrainLinear}. Observe that for each $i \in [N]$,
\begin{align*}
    \inner{\bs{W\obs}_i, \bs{y}_i} &= \sum_{j \in [m]}(\bs{W\obs}_i)_{j}y_{ij} \\
    &= \sum_{j \in [m]}\inner{\bs{w}_{j}, \bs{\obs}_i}y_{ij} \\
    &= \sum_{j \in [m]}\sum_{k \in [d]}\obs_{ik}W_{jk}y_{ij} \\
    &= \sum_{j \in [m]}\sum_{k \in [d]}\obs_{ik}\left(\frac{1}{2}(W_{jk} + y_{ij})^{2} - \frac{1}{2}(W_{jk}^{2} + y_{ij}^{2})\right).
\end{align*}
Then the objective function in \eqref{prob:fixedPrimalTrainLinear} can be written as
\begin{align*}
    \frac{1}{N}&\sum_{i \in [N]}(\inner{\bs{c}, \bs{x}_i^{\star}} - \inner{\bs{W\obs}_i, \bs{y}_i})  + \lambda \, r(\bs{W}) + \gamma \, f(\bs{W}) \\
    &= \frac{1}{N}\sum_{i \in [N]}\left(\inner{\bs{c}, \bs{x}_i^{\star}} - \sum_{j \in [m]}\sum_{k \in [d]}\obs_{ik}\left(\frac{1}{2}(W_{jk} + y_{ij})^{2} - \frac{1}{2}(W_{jk}^{2} + y_{ij}^{2})\right)\right) + \lambda \, r(\bs{W}) + \gamma \, f(\bs{W}) \\
    &= \small{\underbrace{\frac{1}{N}\sum_{i \in [N]}\left(\inner{\bs{c}, \bs{x}_i^{\star}} + \frac{1}{2}\sum_{j \in [m]}\left[\sum_{k \in K^{+}(i)}\obs_{ik}(W_{jk}^{2} + y_{ij}^{2}) - \sum_{k \in K^{-}(i)}\obs_{ik}(W_{jk} + y_{ij})^{2}\right]\right) + \lambda \, r(\bs{W}) + \gamma \, f(\bs{W})}_{F_{1}(\bs{W}, (\bs{y}_i))}} \\
    & \quad - \underbrace{\frac{1}{2N}\sum_{i \in [N]}\sum_{j \in [m]}\left(\sum_{k \in K^{+}(i)}\obs_{ik}\left(W_{jk} + y_{ij}\right)^{2} - \sum_{k \in K^{-}(i)}\obs_{ik}\left(W_{jk}^{2} + y_{ij}^{2}\right)\right)}_{F_{2}(\bs{W}, (\bs{y}_i))},
\end{align*}
where $K^{-}(i) \coloneqq \{k \in [d] ~ | ~ \obs_{ik} < 0\}$ and $K^{+}(i) \coloneqq \{k \in [d] ~ | ~ \obs_{ik} > 0\}$. This shows that $F_{1}$ and $F_{2}$ are convex functions. We can solve problem \eqref{prob:fixedPrimalTrainLinear} with this DC representation $F_1 - F_2$ of the objective function using the Convex-Concave Procedure (CCP) of \cite{yuille2003concave}. The basic idea is that at each iteration, we solve a convexified version of problem \eqref{prob:fixedPrimalTrainLinear} consisting of $F_1$ and the first-order approximation of the function $F_2$. This requires the computation of the gradient $\nabla F_2$. It is easy to see that
\begin{equation}
    \nabla_{W_{jk}}F_2 = \frac{1}{N}\sum_{i \in [N]}\left(\obs_{ik}(W_{jk} + y_{ij})\right)
\end{equation}
if $k \in K^+(i)$ and
\begin{equation}
    \nabla_{W_{jk}}F_2 = -\frac{1}{N}\sum_{i \in [N]}\obs_{ik}W_{jk}
\end{equation}
if $k \in K^-(i)$. Alternatively,
\begin{equation}
    \nabla_{y_{ij}}F_2 = \frac{1}{N}\left(y_{ij}\sum_{k \in [d]}|\obs_{ik}| + \sum_{k \in K^+(i)}\obs_{ik}W_{jk}\right).
\end{equation}
To write the algorithm, we perform a change of variables 
\begin{equation*}
    {\bs{u} = (w_{11}, \ldots, w_{1d}, \ldots, w_{m_1}, \ldots, w_{md}, y_{11}, \ldots, y_{1m}, \ldots, y_{N1}, \ldots, y_{Nm})}
\end{equation*}
and denote by $\bs{u}_{l_1 : l_2}$ the subvector $(u_{l_1}, \ldots, u_{l_2})$. Following \cite[Algorithm~1.1]{lipp2016variations}, we present Algorithm \ref{alg:ccp}.
\begin{algorithm}
\caption{Convex-Concave Procedure \label{alg:ccp}}
\begin{algorithmic}[1]
\State Parameters: $\lambda, \gamma > 0$;
\State Initialize $\bs{W}^{t}$, $(\bs{y}_i)^t$, and $t=0$;
\While{termination criteria are not satisfied}
\State Solve the convexified subproblem
\begin{align}
    \min_{\bs{u}} \quad & F_1(\bs{u}) - \left(F_2(\bs{u}^t) + \inner{\nabla F_2(\bs{u}^t), \bs{u} - \bs{u}^t}\right) \nonumber \\
    \text{s.t.} \quad & \inner{\bs{a}_j, \bs{x}_i^\star} \geq \inner{\bs{u}_{j : j + d}, \bs{\obs}_i}, ~ \forall i \in [N], ~ \forall j \in [m], \notag \\
    & \bs{A}^\top \bs{u}_{(d + i - 1)m + 1 : (d + i)m} \leq \bs{c}, ~ \forall i \in [N], \notag \\
    & \bs{u}_{(d + i - 1)m + 1 : (d + i)m} \geq \bs{0}, ~ \forall i \in [N].
    \end{align}
\State $t \gets t + 1$;
\EndWhile
\State \Return $(\widehat{\bs{W}}, (\hat{\bs{y}}_i)) = (\bs{W}^t, (\bs{y}_i)^t)$
\end{algorithmic}
\end{algorithm}

\section{Details and Additional Results for Synthetic Data Experiment}\label{app:details_and_additional_results_synthetic_data_experiment}

\textit{Data generation}: We generate a cost vector $\bs{c} \in \RR^n$ and a constraint matrix $\bs{A} \in \RR^{m \times n}$ with components $\bs{c}_l, \bs{A}_{jl} \overset{\text{i.i.d.}}{\sim} \set{U}[-10, 10], l \in [n], j \in [m]$. We generate a ground truth linear model $\bs{W}^\star \in \RR^{m \times d}$ with components $W_{jk}^\star \overset{\text{i.i.d.}}{\sim} \text{Bernoulli}(0.5), j \in [m], k \in [d],$ as in \citep{elmachtoub2022smart}. We generate $\bs{\obs}_{i} \in \RR^d$ wtih components $\bs{\obs}_{ik} \overset{\text{i.i.d.}}{\sim} \set{U}[-10, 10], i \in [N], k \in [d],$ and update $\obs_{i1} \gets \obs_{i1} + 10.1$ to ensure feasibility of the optimistic and primal-DAL problems. Finally, we compute $b_{ij} = \frac{1}{\sqrt{d}}\inner{\bs{w}_j^\star, \bs{\obs}_i} + \epsilon_{ij}, i \in [N], j \in [m]$, where $\epsilon_{ij} \overset{\text{i.i.d.}}{\sim} \set{N}(0, 1)$.

\textit{Comparison methods}: We solve the optimistic-DAL problem \eqref{prob:optimisticTrain} under the hypothesis class of linear models, i.e.,
\begin{align}\label{prob:optimisticTrainLinear}
    \min_{\bs{W}}~& \bigg \{\frac{1}{N} \sum_{i \in [N]} (\inner{\bs{c} , \bs{x}_{i}^{\star}} -  \inner{\bs{W\obs}_i, \bs{y}_{i}^{\star}}) ~\bigg \vert~ \bs{A}\bs{x}_{i}^{\star} \geq \bs{W\obs}_{i} \quad \forall i \in [N]\bigg \}.
\end{align}
We solve the primal-DAL problem \eqref{prob:fixedPrimalTrainLinear} using the convex regularizer $r(\bs{W}) \coloneqq \sum_{j \in [m]}\sum_{k \in [d]}|W_{jk}|$ and the convex penalty function $\phi(\bs{W}) \coloneqq \sum_{i \in [N]}\sum_{j \in [m]}\max\{0, b_{ij} - \inner{\bs{w}_j, \bs{\obs}_i}\}$, which penalizes violation of the overpredictive constraints $\bs{W\obs}_i \geq \bs{b}_i$. Note that the convex functions $r(\bullet)$ and $\phi(\bullet)$ that we choose are a digression from Theorem \ref{thm:conv} as they are nondifferentiable. We present an approximation of the dual-DAL problem \eqref{prob:fixedDualTrain} by applying the method presented in \citep{elmachtoub2022smart}, which introduced a loss function to train a predictor for the case $\hat{\bs{c}} := p(\bs{\obs})$, i.e., the contextual vector is linked to the cost vector of C-LP \eqref{prob:downstream}. We apply the derivation given in the reference to the dual of \eqref{prob:downstream} and obtain the following problem:
\begin{align} \label{prob:fixedDualModTrain}
    \min_{\bs{W}, (\bs{x}_i)}~& \bigg \{\frac{1}{N} \sum_{i \in [N]} (\inner{\bs{c} , \bs{x}_i} -  \inner{\alpha\bs{W\obs}_{i} - \bs{b}_{i}, \bs{y}_{i}^{\star}}) ~\bigg \vert~ \bs{A}\bs{x}_{i} \geq \alpha \bs{W\obs}_i - \bs{b}_i, ~\bs{x}_i \geq \bs{0} \quad \forall i \in [N]\bigg \},
\end{align}
where $\alpha \geq 0$ is a given parameter (see Appendix \S\ref{sec:deriving_fixed_dual_mod_train} for full details of the derivation). Setting $\bs{b}_i = (\alpha - 1) \bs{W} \bs{\obs}_i$, the problem \eqref{prob:fixedDualModTrain} recovers \eqref{prob:fixedDualTrain}.

Regarding the machine learning models, we solve the linear regression problem
\begin{equation} \label{prob:linearRegression}
\min_{\bs{W} \in \mathbb{R}^{m \times d}} \quad ||\mathfrak{X} W^\top - \mathfrak{B}||_F^2,
\end{equation}
where
$\mathfrak{X} = \begin{bmatrix}
\bs{\obs}_{1}^\top\\
\vdots\\
\bs{\obs}_{N}^\top
\end{bmatrix} \in \RR^{N \times d}$,
$\mathfrak{B} = \begin{bmatrix}
\bs{b}_{1}^\top\\
\vdots\\
\bs{b}_{N}^\top
\end{bmatrix} \in \RR^{N \times m}$, and $||\cdot||_F$ denotes the Frobenius norm. We solve the lasso regression problem
\begin{equation}
\label{prob:lassoRegression}
\min_{\bs{W} \in \mathbb{R}^{m \times d}} \quad ||\mathfrak{X} W^\top - \mathfrak{B}||_F^2 + \alpha_{lasso}\sum_{j \in [m]}\sum_{k \in [d]}|W_{jk}|,
\end{equation}
where $\mathfrak{X}$ and $\mathfrak{B}$ are as above and $\alpha_{lasso} \geq 0$ is a hyperparameter.

\textit{Sensitivity Analysis}: We perform a sensitivity analysis for the primal-DAL problem \eqref{prob:fixedPrimalTrainLinear} and the dual-DAL problem \eqref{prob:fixedDualModTrain} as we vary their hyperparameters $(\lambda, \gamma)$ and $\alpha$, respectively. Throughout, we set $N = 1000$ and display the results of 50 replications.

To start, we set $\gamma = 0$ for the primal-DAL problem \eqref{prob:fixedPrimalTrainLinear}. Figure 
{\ref{fig:primal_dal_gamma_0_W_hat_sparsity}} shows the number of zero components in the solution $\widehat{\bs{W}}$ obtained from of Algorithm \ref{alg:acs} as the regularization parameter $\lambda$ varies (recall that $\widehat{\bs{W}}$ has {21} components for the synthetic experiments).
\begin{figure}[h!]
    \centering
    \includegraphics[scale = 0.5]{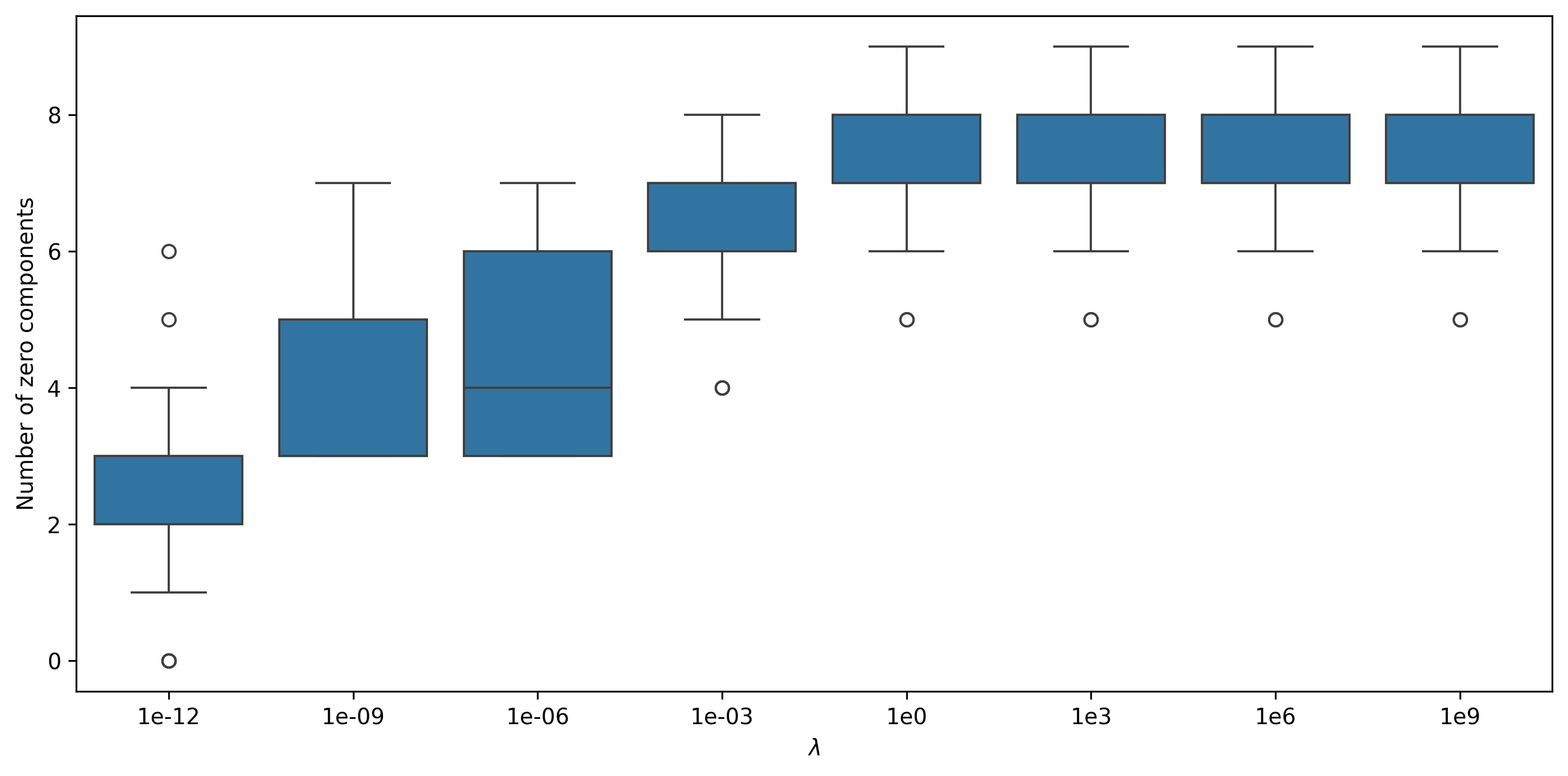}
    \caption{Sparsity of the solution $\widehat{\bs{W}}$ from of Algorithm \ref{alg:acs} when $\gamma = 0$.}
    \label{fig:primal_dal_gamma_0_W_hat_sparsity}
\end{figure}
Obviously, as $\lambda$ increases, so do the number of zero components in the model $\widehat{\bs{W}}$. We also observe the affect of the regularization parameter $\lambda$ on the average in-sample optimality gap $\frac{1}{N}\sum_{i \in [N]}(\inner{\bs{c}, \bs{x}_i^\star} - \inner{\widehat{\bs{W}}\bs{\obs}_i, \hat{\bs{y}}_i})$ as well as the value $r(\widehat{\bs{W}}) = \sum_{j \in [m]}\sum_{k \in [d]}|\widehat{W}_{jk}|$ using the solution $(\widehat{\bs{W}}, (\hat{\bs{y}}_i))$ obtained from Algorithm \ref{alg:acs}. The results are shown in Figure
{\ref{fig:primal_dal_gamma_0_component_function_vals}.}
\begin{figure}[h!]
    \centering
    \includegraphics[scale = 0.5]{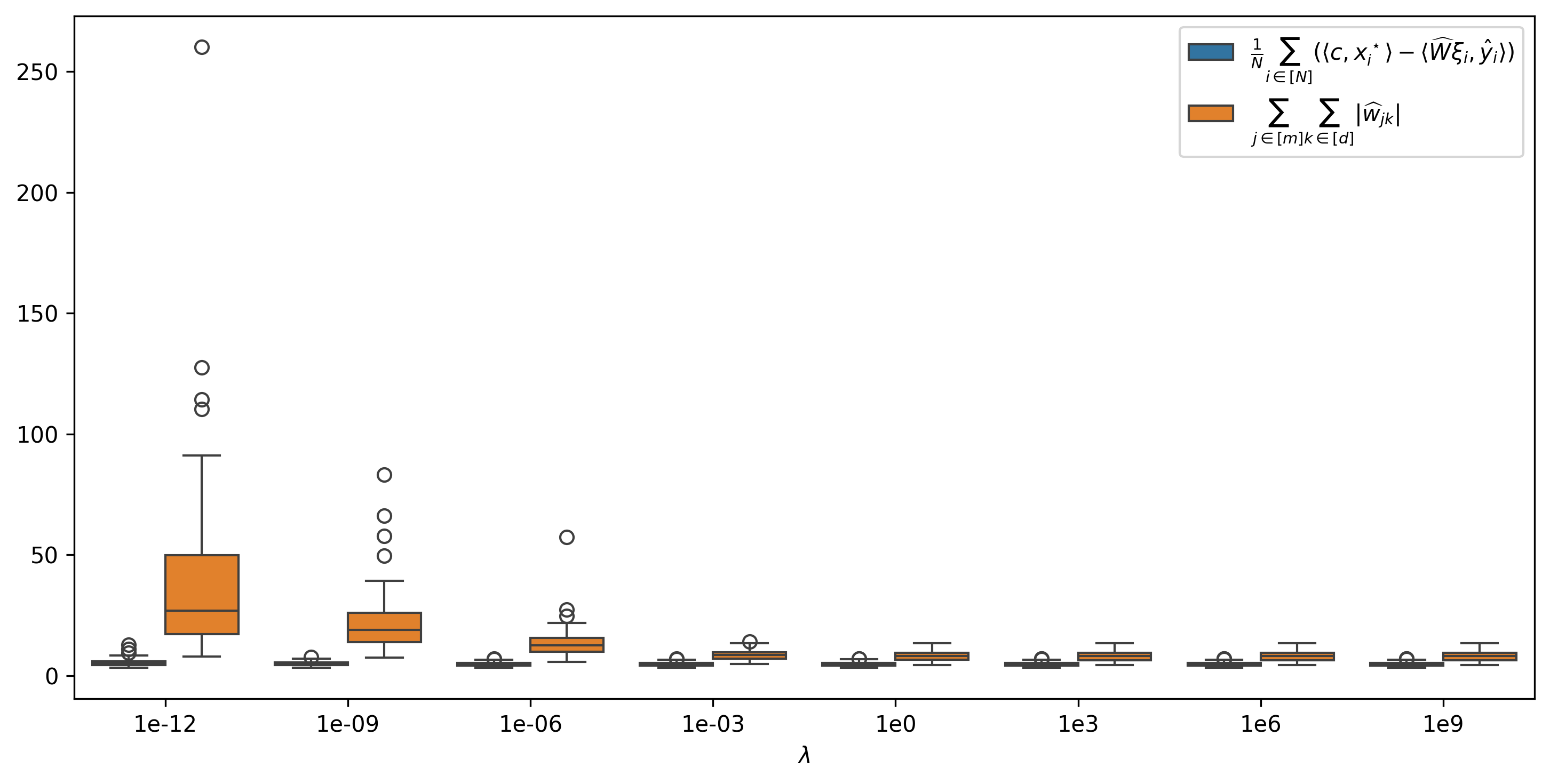}
    \caption{Component function values in primal-DAL problem \eqref{prob:fixedPrimalTrainLinear} evaluated using the solution of Algorithm \ref{alg:acs} ($\gamma = 0$).}
    \label{fig:primal_dal_gamma_0_component_function_vals}
\end{figure}
For $\lambda \in [10^{-12}, 10^{0}]$, we see that 
the regularization level 
decreases {whereas the average in-sample optimality gap remains relatively unaffected}, suggesting that regularization is effective in improving model quality. As $\lambda$ increases further, we see that 
there is minimal impact on the regularization level $r(\widehat{\bs{W}})$.

We also investigate the affect of jointly varying $(\lambda, \gamma)$ in the primal-DAL problem \eqref{prob:fixedPrimalTrainLinear} on the average in-sample optimality gap $\frac{1}{N}\sum_{i \in [N]}(\inner{\bs{c}, \bs{x}_i^\star} - \inner{\widehat{\bs{W}}\bs{\obs}_i, \hat{\bs{y}}_i})$, the value $r(\widehat{\bs{W}}) = \sum_{j \in [m]}\sum_{k \in [d]}|\widehat{W}_{jk}|$, and the value $\phi(\widehat{\bs{W}}) = \sum_{i \in [N]}\sum_{j \in [m]}\max\{0, b_{ij} - \inner{\bs{w}_j, \bs{\obs}_i}\}$ of the penalty function using the solution $(\widehat{\bs{W}}, (\hat{\bs{y}}_i))$ obtained from Algorithm \ref{alg:acs}. We plot the median of each of these values in Figure
{\ref{fig:primal_dal_gamma_pos}}.
\begin{figure}[h!]
  \centering
  \begin{subfigure}[b]{0.32\textwidth}
    \includegraphics[width=\linewidth]{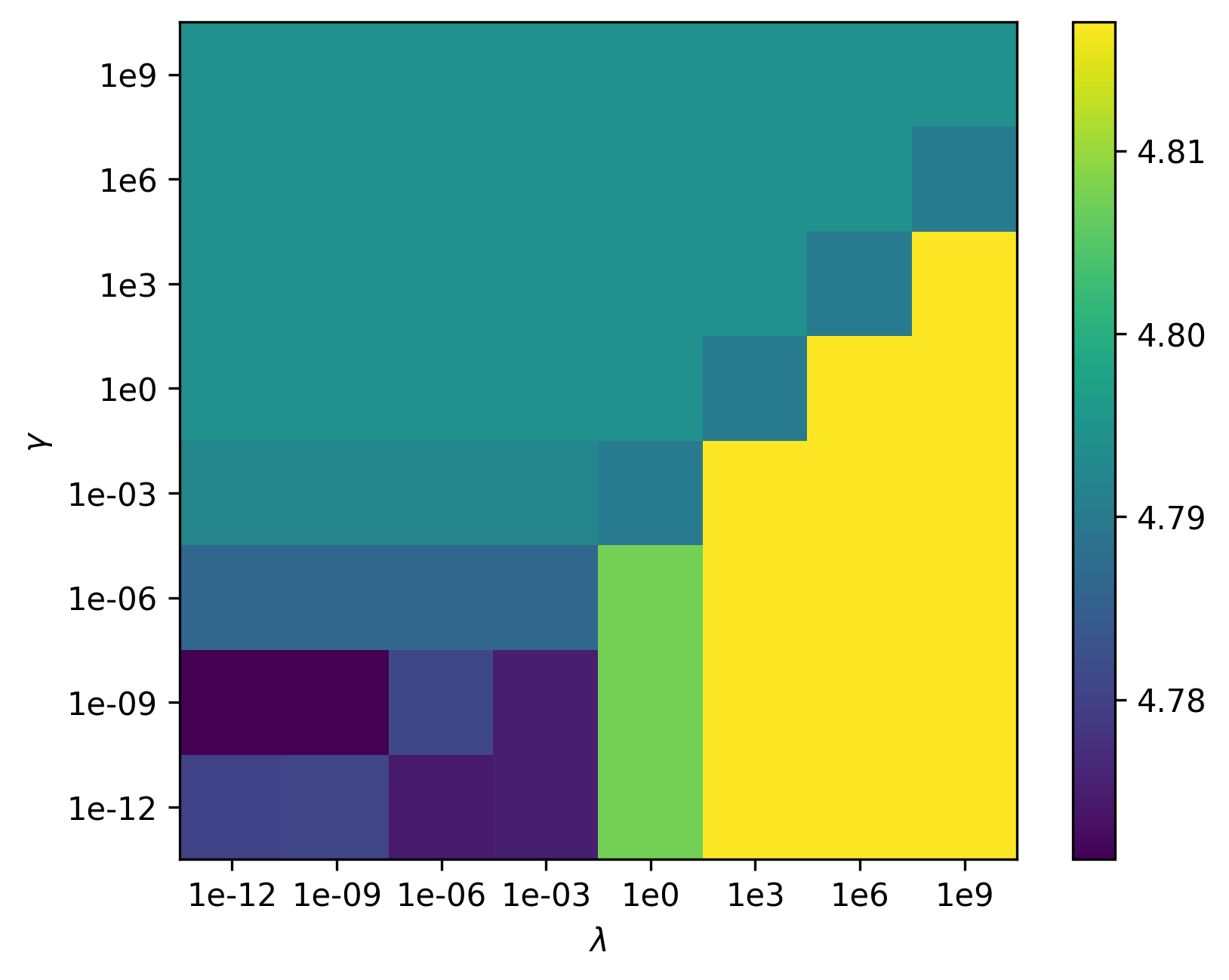}
    \caption{\scriptsize{$\frac{1}{N}\sum_{i \in [N]}(\inner{\bs{c}, \bs{x}_i^\star} - \inner{\widehat{\bs{W}}\bs{\obs}_i, \hat{\bs{y}}_i})$}}
    \label{fig:primal_dal_gamma_pos_obj_val_opt_gap}
  \end{subfigure}
  \hfill
  \begin{subfigure}[b]{0.32\textwidth}
    \includegraphics[width=\linewidth]{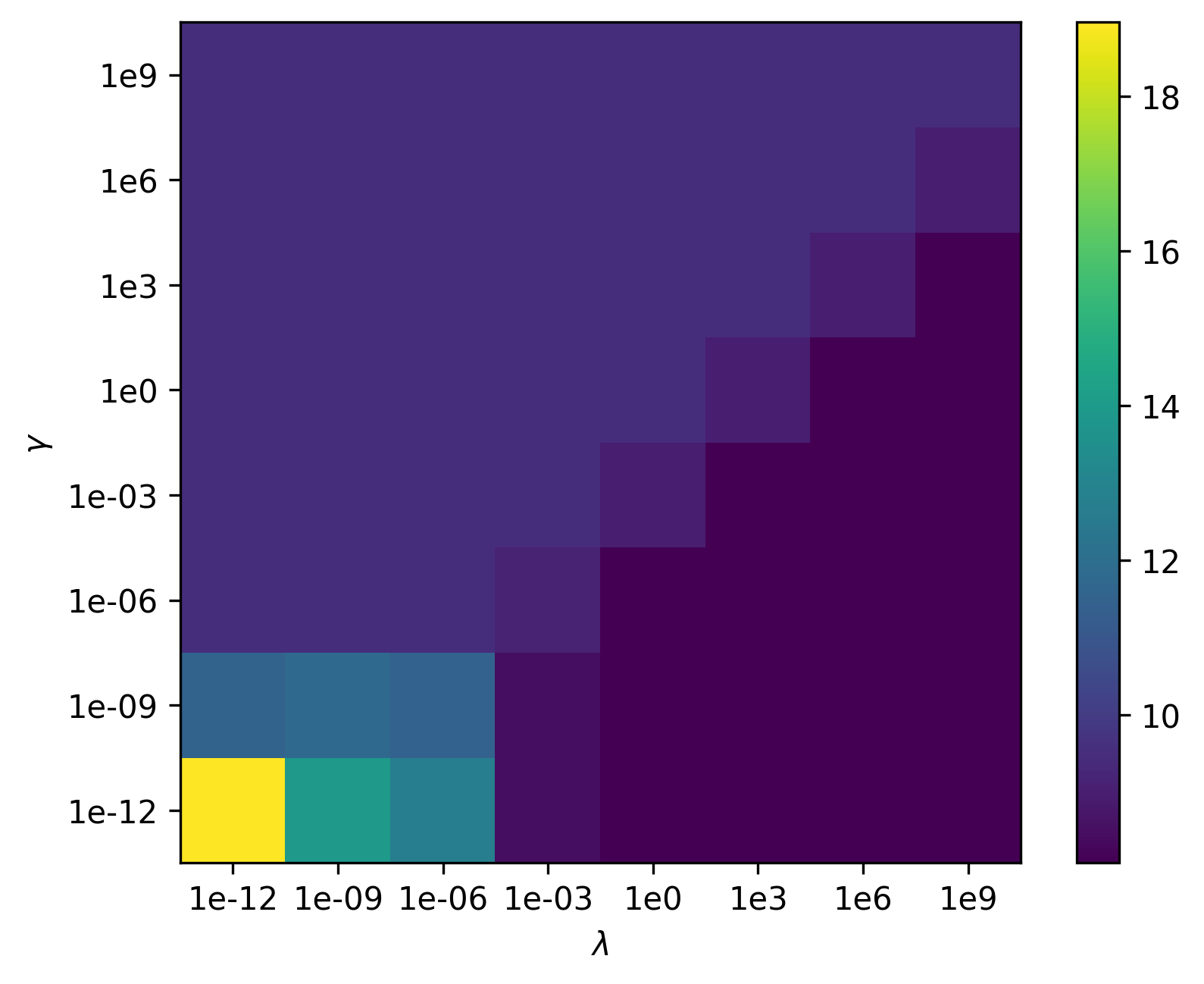}
    \caption{\scriptsize{$\sum_{j \in [m]}\sum_{k \in [d]}|\widehat{W}_{jk}|$}}
    \label{fig:primal_dal_gamma_pos_obj_val_reg}
  \end{subfigure}
  \hfill
  \begin{subfigure}[b]{0.32\textwidth}
    \includegraphics[width=\linewidth]{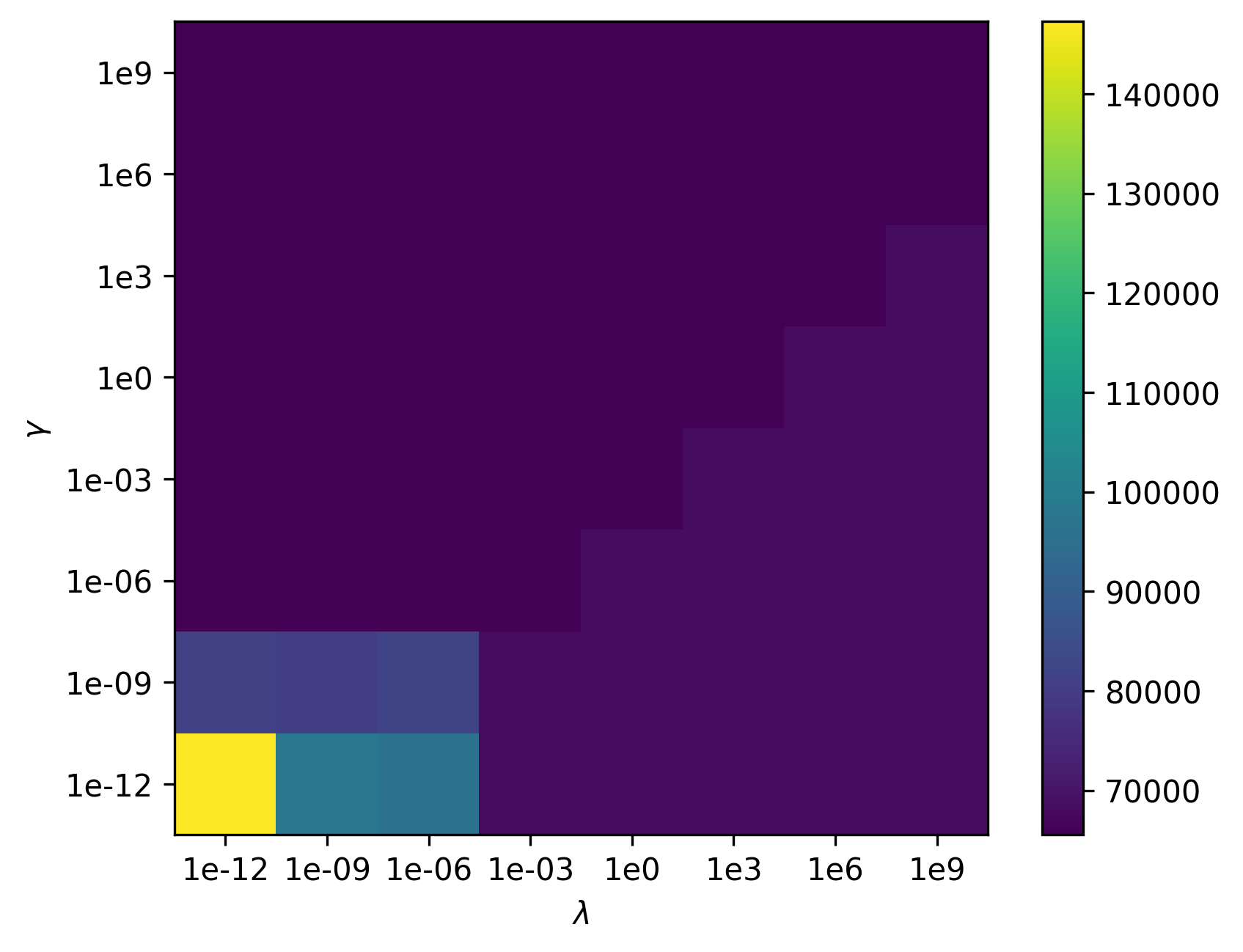}
    \caption{\scriptsize{$\sum_{i \in [N]}\sum_{j \in [m]}\max\{0, b_{ij} - \inner{\bs{w}_j, \bs{\obs}_i}\}$}}
    \label{fig:primal_dal_gamma_pos_obj_val_pen}
  \end{subfigure}
  \caption{Affect of $(\lambda, \gamma)$ on the component function values of primal-DAL problem \eqref{prob:fixedPrimalTrainLinear} when $\gamma > 0$ }
  \label{fig:primal_dal_gamma_pos}
\end{figure}
Regardless of the choice of $(\lambda, \gamma)$, we see that the dominating term is the penalty function $\phi$. One interesting observation is that for a fixed regularization parameter $\lambda \in \{10^0, 10^3, 10^6, 10^9\}$, the average in-sample duality gap (part (a)) goes down as the penalty parameter $\gamma$ increases. This seems to coincide with Corollary \ref{cor:squeezing}, as the constraints $\bs{Ax}_i^\star \geq \bs{W\obs}_i$ are enforced in \eqref{prob:fixedPrimalTrainLinear} and the constraints $\bs{W\obs}_i \geq \bs{b}_i$ are penalized when violated through the function $\phi$.

Finally, we observe the affect of varying the parameter $\alpha$ on the solution $(\widehat{\bs{W}}, (\hat{\bs{x}}_i))$ of the dual-DAL problem \eqref{prob:fixedDualModTrain}. We compute the optimal value $\frac{1}{N}\sum_{i \in [N]}(\inner{\bs{c}, \hat{\bs{x}}_i} - \inner{\alpha\widehat{\bs{W}}\bs{\obs}_i - \bs{b}_i, \bs{y}_i^\star})$ of \eqref{prob:fixedDualModTrain} as well as the average in-sample optimality gap $\frac{1}{N}\sum_{i \in [N]}|\inner{\bs{c}, \bs{x}_i^\star} - \inner{\widehat{\bs{W}}\bs{\obs}_i, \bs{y}_i^\star}|$. Note that the latter value requires absolute values on the summands as the differences $\inner{\bs{c}, \bs{x}_i^\star} - \inner{\widehat{\bs{W}}\bs{\obs}_i, \bs{y}_i^\star}$ may be negative.
We plot the median of these values in Figure
{\ref{fig:dual_dal_component_function_vals}}.
\begin{figure}[h!]
    \centering
    \includegraphics[scale = 0.5]{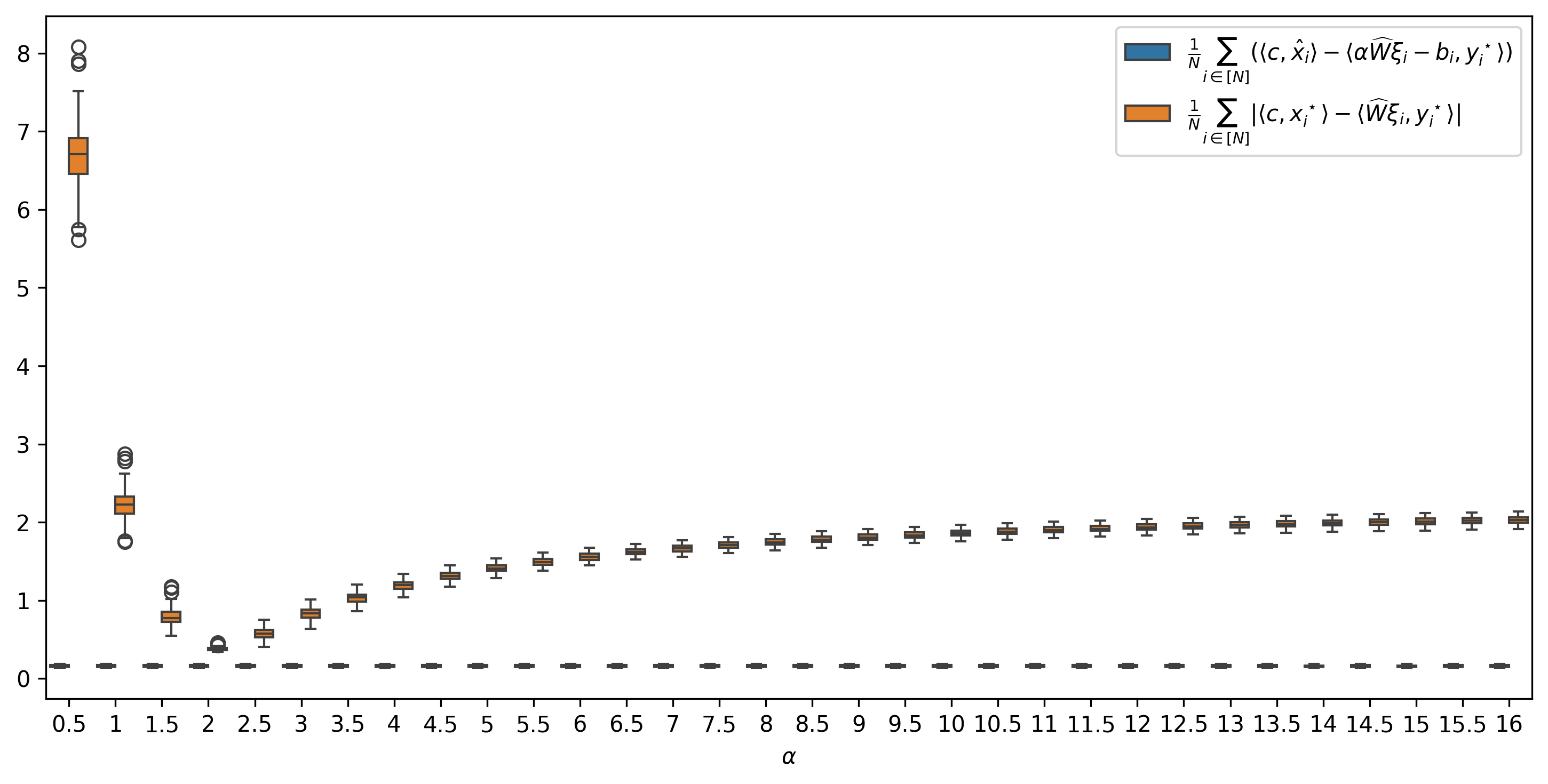}
    \caption{Affect of $\alpha$ on the solution $(\widehat{\bs{W}}, (\hat{\bs{x}}_i))$ of the dual-DAL problem \eqref{prob:fixedDualModTrain}}
    \label{fig:dual_dal_component_function_vals}
\end{figure}
We see that the optimal value of $\eqref{prob:fixedDualModTrain}$ is the same regardless of $\alpha$. However, the average in-sample optimality gap $\frac{1}{N}\sum_{i \in [N]}|\inner{\bs{c}, \bs{x}_i^\star} - \inner{\widehat{\bs{W}}\bs{\obs}_i, \bs{y}_i^\star}|$ is large for small values of $\alpha$, achieves a minimum at $\alpha = 2$, and then slowly increases and tapers off as $\alpha$ increases. 

\textit{Solution Method Comparison in Solving Primal-DAL Problem \eqref{prob:fixedPrimalTrainLinear}}: We compare the performance of various solution methods in solving problem \eqref{prob:fixedPrimalTrainLinear}. We use Algorithms \ref{alg:acs} and \ref{alg:ccp}, each with an upper limit of 100 iterations and the function value decrease condition $\frac{\text{obj}^t - \text{obj}^{t + 1}}{\text{obj}^t} < 0.01$ as the termination criterion,
where ``obj'' denotes the objective function value and the superscript is the iteration number. We also use Gurobi's nonconvex solver, which solves problem \eqref{prob:fixedPrimalTrainLinear} as a mixed integer program. We limit Gurobi to 1 hour of solve time. We fix the training dataset size at $N = 1000$. Table \ref{table:fixedPrimalSolutionComparison} shows the median objective function value and runtime of the various solution methods.
\begin{table}[h!]
\small
\centering
\begin{tabular}{l|cc|cc|}
 & \multicolumn{2}{c|}{$(\lambda, \gamma) = (10^{-3}, 0)$} & \multicolumn{2}{c|}{$(\lambda, \gamma) = (10^{-3}, 10^{-3})$} \\
\cline{2-3} \cline{4-5}
Solution Method & Obj. Value & Time (s) & Obj. Value & Time (s) \\
\hline
Algorithm \ref{alg:acs} & {4.78} &  {51.51} & {70.22} &  {36.32} \\
Algorithm \ref{alg:ccp} & {7.73} & {1062.45} &  {70.95} & {1008.97} \\
Gurobi & {4.80} &  {3613.07} & {70.22} &  {3612.05} \\
\hline
\end{tabular}
\caption{Comparison of solution methods in solving primal-DAL problem \eqref{prob:fixedPrimalTrainLinear}.}
\label{table:fixedPrimalSolutionComparison}
\end{table} 
Algorithm \ref{alg:acs} and Gurobi achieve roughly the same objective function value {under both sets of parameters}, however Algorithm \ref{alg:acs} is 2 orders of magnitude faster. Algorithm \ref{alg:ccp} performs the worst in terms of the objective function value, and its runtime is somewhere between that of the other two solution methods.

\textit{Hyperparameter Tuning}: In the absence of additional constraints on the model $\bs{W}$, we tune the primal-DAL problem \eqref{prob:fixedPrimalTrainLinear} with candidates $(\lambda, \gamma) \in \{10^{-12}, 10^{-9}, 10^{-6}, 10^{-3}, 10^{0}, 10^{3}\} \times \{0\}$ using the feasibility metric $\chi\{\bs{Ax}_i^\star \geq \bs{W\obs}\}$ (higher is better). In the presence of the additional constraints $\bs{W\obs}_i \geq \bs{b}_i, ~ \forall i \in [N],$ we tune \eqref{prob:fixedPrimalTrainLinear} using the same metric and candidates $(\lambda, \gamma) \in {\{10^{-12}, 10^{-6}, 10^{0}, 10^{6}\} \times \{10^5, 10^6, 10^7\}}$. We tune the hyperparameter $\alpha$ in the dual-DAL problem \eqref{prob:fixedDualModTrain} using candidate values $\alpha \in \{0.5, 1, 1.5, 2, 2.5, 3\}$. Lastly, we tune $\alpha_{lasso}$ in the lasso regression problem using sum of squared prediction errors as the metric (lower is better) with candidate values $\{1, 3, 5, 7\}$.

\textit{Additional Results}: Table 
{\ref{tab:synthetic_runtime}} shows the median training time of the different solution methods as the training dataset size $N$ increases.
\begin{table}[h!]
\small
\centering
\begin{tabular}{cccccccc}
    $N$ & Optimistic-DAL & Primal-DAL & Primal-DAL (w/ penalty) & Dual-DAL & LR & Lasso & RF \\ \hline
    250 & 0.12 & 8.29 & 8.13 & 0.41 & $< 0.01$ & $< 0.01$ & 0.18 \\
    500 & 0.23 & 17.38 & 16.89 & 0.92 & $< 0.01$ & $< 0.01$ & 0.27 \\
    750 & 0.29 & 26.82 & 26.78 & 2.81 & $< 0.01$ & $< 0.01$ & 0.32 \\
    1000 & 0.38 & 36.11 & 36.30 & 4.61 & $< 0.01$ & $< 0.01$ & 0.37 \\
    \hline
\end{tabular}
\caption{Training time of different learning problems (in seconds)}
\label{tab:synthetic_runtime}
\end{table}
The DUL models typically solve within 1 second, which is similar to the optimistic and dual-DAL problems (which are both LPs). The nonconvex primal-DAL problem takes slightly longer to solve, but is still generally within 1 minute. {Table \ref{tab:percent_feas_pred_soln_synthetic} shows the percentage of the solutions $\hat{x}_i$ induced by the different prediction models which reside in the corresponding true feasible region. The optimistic- and primal-DAL models achieve a very low $(<1\%)$ value for this metric, and the dual-DAL is slightly better at around $3-6\%$. The standard regression models perform the best at around $18-20\%$. Note that the poor performance of our models for this metric is to be expected, as they are designed to generate predictions that recover the true optimal solution $\bs{x}^\star$, which is inherently at odds with this metric.}
\begin{table}[h!]
\footnotesize
\centering \renewcommand{\arraystretch}{1.2}
\begin{tabular}{cccccccc}
$N$ & Optimistic-DAL & Primal-DAL & Primal-DAL (w/ penalty) & Dual-DAL & {LR} & {Lasso} & {RF} \\
\hline
    250 & 0.46 & 0.36 & 0.41 & 5.73 & 20.42 & 18.75 & 20.24 \\
    500 & 0.26 & 0.25 & 0.26 & 3.57 & 20.61 & 18.84 & 20.36 \\
    750 & 0.17 & 0.17 & 0.17 & 3.21 & 20.50 & 18.66 & 19.89 \\
    1000 & 0.15 & 0.14 & 0.15 & 3.22 & 20.38 & 18.66 & 20.41 \\
\hline
\end{tabular}
\caption{Percentage of solutions $\hat{\bs{x}}_i$ in the validation dataset which are in the true feasible regions.}
\label{tab:percent_feas_pred_soln_synthetic}
\end{table}

\section{Deriving Problem  \eqref{prob:fixedDualModTrain}}\label{sec:deriving_fixed_dual_mod_train}
Consider the C-LP \eqref{prob:downstream} with feasible region $\set{X}(\bs{b})$, where $\bs{b} \in \set{B}$ is arbitrary. Recall that the dual feasible region is $\set{Y} \coloneqq \{\bs{y} \geq \bs{0} ~|~\bs{A}^\top \bs{y} \leq \bs{c}\}$. Denote the set of dual optimal solutions corresponding to $\bs{b}$ as $\set{Y}^\star(\bs{b}) \coloneqq \arg\max_{\bs{y} \geq \bs{0}}\{\inner{\bs{b}, \bs{y}} ~ | ~ \bs{A}^\top \bs{y} \leq \bs{c}\}$ and the corresponding optimal cost as $v^\star(\bs{b})$, i.e., $v^\star(\bs{b}) = \inner{\bs{b}, \bs{y}^\star(\bs{b})}$ for any $\bs{y}^\star(\bs{b}) \in \set{Y}^\star(\bs{b})$; unlike before, here we explicitly show the dependence on $\bs{b}$ for clarity of the derivation.

\begin{definition}[RSPO loss]
Given $\bs{b}$ and a prediction $\hat{\bs{b}}$, the \textit{right-hand side smart predict-then-optimize (RSPO) loss function $\ell_{\text{RSPO}}^{\bs{y}^\star}(\hat{\bs{b}}, \bs{b})$} with respect to $\bs{y}^\star$ is defined as $\ell_{\text{RSPO}}^{\bs{y}^\star}(\hat{\bs{b}}, \bs{b}) \coloneqq v^\star(\bs{b}) - \inner{\bs{b}, \bs{y}^\star(\hat{\bs{b}})}$.
\label{def:rspo_loss}
\end{definition}

A notable drawback of this definition is the dependence on the optimization oracle $\bs{y}^\star$. We consider a variant of the RSPO loss which takes the worst-case solution among vectors $\bs{y} \in \set{Y}^\star(\hat{\bs{b}})$.

\begin{definition}[Unambiguous RSPO loss]
    Given $\bs{b}$ and a prediction $\hat{\bs{b}}$, the \textit{unambiguous RSPO loss function $\ell_{\text{RSPO}}(\hat{\bs{b}}, \bs{b})$} is defined as $\ell_{\text{RSPO}}(\hat{\bs{b}}, \bs{b}) \coloneqq v^\star(\bs{b}) - \min_{\bs{y} \in \set{Y}^\star(\hat{\bs{b}})}\inner{\bs{b}, \bs{y}}$.
\end{definition}

For a fixed right-hand-side vector $\bs{b}$, the RSPO loss may may not be continuous in $\hat{\bs{b}}$ because $\bs{y}^\star(\hat{\bs{b}})$ (and the set $\set{Y}^\star(\hat{\bs{b}})$) may not be continuous in $\hat{\bs{b}}$. Similar to \cite{elmachtoub2022smart}, we will derive a tractable surrogate loss function for $\ell_{RSPO}(\cdot, \cdot)$. To that end, consider a parameter $\alpha \geq 0$ and note that
\begin{equation*}
    \ell_{RSPO}(\hat{\bs{b}}, \bs{b}) = v^\star(\bs{b}) - \min_{\bs{y} \in \set{Y}^\star(\hat{\bs{b}})}\{\inner{\bs{b}, \bs{y}} - \inner{\alpha \hat{\bs{b}}, \bs{y}}\} - \alpha v^\star(\hat{\bs{b}})
\end{equation*}
which follows from the fact that $v^\star(\hat{\bs{b}}) = \inner{\hat{\bs{b}}, \bs{y}}$ for all $\bs{y} \in \set{Y}^\star(\hat{\bs{b}})$. We can replace the constraint $\bs{y} \in \set{Y}^\star(\hat{\bs{b}})$ with $\bs{y} \in \set{Y}$ to obtain an upper bound. Since this is true for any $\alpha \geq 0$, it follows that
\begin{align}
    \ell_{RSPO}(\hat{\bs{b}}, \bs{b}) \leq & \inf_{\alpha \geq 0}\left\{v^\star(\bs{b}) - \min_{\bs{y} \in \set{Y}}\{\inner{\bs{b}, \bs{y}} - \inner{\alpha \hat{\bs{b}}, \bs{y}}\} - \alpha v^\star(\hat{\bs{b}})\right\} \nonumber \\
    & = v^\star(\bs{b}) + \inf_{\alpha \geq 0}\left\{\max_{\bs{y} \in \set{Y}}\{\inner{\alpha \hat{\bs{b}}, \bs{y}} - \inner{\bs{b}, \bs{y}}\} - \alpha v^\star(\hat{\bs{b}})\right\}.
    \label{ineq:rhs_alpha_ub}
\end{align}

In fact, inequality \eqref{ineq:rhs_alpha_ub} can be shown to be an equality using duality theory, and the optimal value of $\alpha$ tends to $\infty$.

\begin{proposition}\label{prop:rspo_representation}
Given $\bs{b}$ and a prediction $\hat{\bs{b}}$, the function $\alpha \mapsto \max_{\bs{y} \in \set{Y}}\{\inner{\alpha \hat{\bs{b}}, \bs{y}} - \inner{\bs{b}, \bs{y}}\} - \alpha v^\star(\hat{\bs{b}})$ is monotone decreasing on $\RR$, and the RSPO loss may be represented as $\ell_{RSPO}(\hat{\bs{b}}, \bs{b}) = v^\star(\bs{b}) + \lim_{\alpha \to \infty}\left\{\max_{\bs{y} \in \set{Y}}\{\inner{\alpha \hat{\bs{b}}, \bs{y}} - \inner{\bs{b}, \bs{y}}\} - \alpha v^\star(\hat{\bs{b}})\right\}$.
\end{proposition}

\proof{Proof of Proposition \ref{prop:rspo_representation}:}
    See the proof of Proposition 2 in \citep{elmachtoub2022smart}.

\endproof

Using an arbitrary hypothesis class $\set{P}$ of prediction functions, the loss function $\ell_{RSPO}$ as given in Proposition \ref{prop:rspo_representation}, and a dataset $\set{D}_N = \{(\bs{\obs}_i, \bs{b}_i)\}_{i \in [N]}$ of observations sampled independently from $\set{B} \times \rvset$, we have that
\begin{align}
    & \min_{p \in \set{P}}\frac{1}{N}\sum_{i = 1}^N\ell_{RSPO}(p(\bs{\obs}_i), \bs{b}_i) \nonumber\\
    & = \min_{p \in \set{P}} \frac{1}{N}\sum_{i = 1}^N\Bigg[v^\star(\bs{b}_i) + \lim_{\alpha_i \to \infty}\left\{\max_{\bs{y} \in \set{Y}}\{\inner{\alpha_i p(\bs{\obs}_i), \bs{y}} - \inner{\bs{b}_i, \bs{y}}\} - \alpha_i v^\star(p(\bs{\obs}_i))\right\} \Bigg]\nonumber\\
    & = \min_{p \in \set{P}} \frac{1}{N}\sum_{i = 1}^N\Bigg[v^\star(\bs{b}_i) + \lim_{\alpha_i \to \infty}\left\{\max_{\bs{y} \in \set{Y}}\{\inner{\alpha_i p(\bs{\obs}_i), \bs{y}} - \inner{\bs{b}_i, \bs{y}}\} - \inner{\alpha_i p(\bs{\obs}_i), \bs{y}^\star(\alpha_i p(\bs{\obs}_i))}\right\} \Bigg]\nonumber\\
    & = \min_{p \in \set{P}} \frac{1}{N}\lim_{\alpha \to \infty}\sum_{i = 1}^N\Bigg[v^\star(\bs{b}_i) + \max_{\bs{y} \in \set{Y}}\{\inner{\alpha p(\bs{\obs}_i), \bs{y}} - \inner{\bs{b}_i, \bs{y}}\} - \inner{\alpha p(\bs{\obs}_i), \bs{y}^\star(\alpha p(\bs{\obs}_i))} \Bigg]\nonumber\\
    & \leq \min_{p \in \set{P}} \frac{1}{N}\sum_{i = 1}^N\Bigg[v^\star(\bs{b}_i) + \max_{\bs{y} \in \set{Y}}\{\inner{\alpha' p(\bs{\obs}_i), \bs{y}} - \inner{\bs{b}_i, \bs{y}}\} - \inner{\alpha' p(\bs{\obs}_i), \bs{y}^\star(\alpha' p(\bs{\obs}_i))} \Bigg]\nonumber\\
    & \leq \min_{p \in \set{P}} \frac{1}{N}\sum_{i = 1}^N\Bigg[v^\star(\bs{b}_i) + \max_{\bs{y} \in \set{Y}}\{\inner{\alpha' p(\bs{\obs}_i), \bs{y}} - \inner{\bs{b}_i, \bs{y}}\} - \inner{\alpha' p(\bs{\obs}_i), \bs{y}^\star(\bs{b}_i)} \Bigg].
    \label{ineq:rhs_spo_plus_derivation}
\end{align}
where $\alpha' \geq 0$ is arbitrary. Note that the first equality holds by Proposition \ref{prop:rspo_representation}; the second equality holds since for any positive scalar $\alpha, \alpha v^\star(\bs{b}) = v^\star(\alpha \bs{b}) = (\alpha \bs{b})^{\top}\bs{y}^\star(\alpha \bs{b})$; the third equality holds since all $\alpha_i$ tend towards $\infty$; the first inequality holds from inequality \eqref{ineq:rhs_alpha_ub} with $\alpha' \geq 0$; and the second inequality holds since $\bs{y}^\star(\bs{b}_i)$ is feasible to problem the dual problem with the cost vector $\alpha' p(\bs{\obs}_i)$. We now arrive at the definition of the RSPO+ loss function, which is exactly the summand in \eqref{ineq:rhs_spo_plus_derivation}.
\begin{definition}[RSPO+ loss]
    Given $\bs{b}$ and a prediction $\hat{\bs{b}}$, the \textit{RSPO+ loss function $\ell_{\text{RSPO+}}^\alpha(\hat{\bs{b}}, \bs{b})$} is defined as $\ell_{\text{RSPO+}}^\alpha(\hat{\bs{b}}, \bs{b}) \coloneqq v^\star(\bs{b}) + \max_{\bs{y} \in \set{Y}}\left\{\inner{\alpha\hat{\bs{b}}, \bs{y}} - \inner{\bs{b}, \bs{y}}\right\} - \inner{\alpha\hat{\bs{b}}, \bs{y}^\star(\bs{b})}$ where $\alpha \geq 0$ is an input parameter.
    \label{def:rhs_true_spo}
\end{definition}

To finish with the derivation, we see by linear programming strong duality that
\begin{align*}
    \ell_{RSPO+}^\alpha(\bs{W\obs}_i, \bs{b}_i) & = \max_{\bs{y} \in \set{Y}}\left\{\inner{\alpha\bs{W\obs}_i - \bs{b}_i, \bs{y}}\right\} - \inner{\alpha\bs{W\obs}_i - \bs{b}_i, \bs{y}^\star(\bs{b}_i)} \\
    & = \min_{\bs{x}_i \geq \bs{0}}\{\inner{\bs{c}, \bs{x}_i} ~ | ~ \bs{Ax}_i \geq \alpha\bs{W\obs}_i - \bs{b}_i\} - \inner{\alpha\bs{W\obs}_i - \bs{b}_i, \bs{y}^\star(\bs{b}_i)}.
\end{align*}
Hence, the empirical risk minimization problem $\min_{\bs{W}}\frac{1}{N}\sum_{i \in [N]}\ell_{RSPO+}^\alpha(\bs{W\obs}_i, \bs{b}_i)$ can be written as
\begin{align*}
    \min_{\bs{W}, (\bs{x}_i)} \quad & \frac{1}{N} \sum_{i = 1}^N (\inner{\bs{c}, \bs{x}_i} - \inner{\alpha \bs{W\obs}_i - \bs{b}_i, \bs{y}^\star(\bs{b}_i)}) \\
    \text{s.t.} \quad & \bs{Ax}_i \geq \alpha \bs{W\obs}_i - \bs{b}_i, ~ \forall i \in [N], \\
    & \bs{x}_i \geq \bs{0}, ~ \forall i \in [N].
\end{align*}

\section{Details of Network Optimization Experiment}\label{sec:network_optimization_details}

\textit{Network Optimization Problem}: We consider a network optimization problem defined by the following: a set $\set{F}$ of factories, a set $\set{H}$ of warehouses, and a set $\set{S}$ of stores. Units of some arbitrary good must travel from factories to warehouses, and then to the stores, where demand is realized. We assume that there is an edge in the network between each factory/warehouse as well as each warehouse/store. We denote by $c_{fh}^1$ the unit shipping cost from factory $f$ to warehouse $h$, and $c_{hs}^2$ the unit shipping cost from warehouse $h$ to store $s$. We allow for demand to be met at a store $s$ from an external supplier, at a unit cost of $\beta > \max_{f \in \set{F}, h \in \set{H}}c_{fh}^1 + \max_{h \in \set{H}, s \in \set{S}}c_{hs}^2$. We assume that there is a capacity of $M$ units of the good which may be processed at each warehouse. Lastly, we denote by $\tilde{d}_s$ the uncertain demand for the good at store $s$. We define decision variables $x_{fh}^1$ as the number of units to ship from factory $f$ to warehouse $h$, $x_{hs}^2$ as the number of units to ship from warehouse $h$ to store $s$, and $x_s^3$ as the number of units to purchase from an external source to send to store $s$. Using this data, we write the network optimization problem as
\begin{subequations}\label{prob:network}
\begin{align}
\min_{\bs{x}^1, \bs{x}^2, \bs{x}^3} \quad & \sum_{f \in \set{F}}\sum_{h \in \set{H}}c_{fh}^1 x_{fh}^1 + \sum_{h \in \set{H}}\sum_{s \in \set{S}}c_{hs}^2 x_{hs}^2 + \beta \sum_{s \in \set{S}}x_s^3 \\
\text{s.t.} \quad & \sum_{f \in \set{F}}x_{fh}^1 = \sum_{s \in \set{S}}x_{hs}^2, ~ \forall h \in \set{H}, \label{constr:flow_balance}\\
\quad & \sum_{f \in \set{F}}x_{fh}^1 \leq M, ~ \forall h \in \set{H}, \label{constr:warehouse_capacity}\\
\quad & x_s^3 \leq \frac{1}{2}\sum_{h \in \set{H}}x_{hs}^2, ~ \forall s \in \set{S}, \label{constr:capacity_on_external_purchasing}\\
\quad & \sum_{h \in \set{H}}x_{hs}^2 + x_s^3 \geq \tilde{d}_s, ~ \forall s \in \set{S}, \label{constr:demand}\\
\quad & \bs{x}^1, \bs{x}^2, \bs{x}^3 \geq \bs{0}.
\end{align}
\end{subequations}
The objective is to minimize total cost, i.e., distribution costs along the network and costs incurred from an external supplier. Constraint \eqref{constr:flow_balance} is a flow balance constraint at the warehouses, whereas constraint \eqref{constr:warehouse_capacity} is a capacity constraint at the warehouses. Constraint \eqref{constr:capacity_on_external_purchasing} sets an upper bound on the number of units that can be purchased from an external supplier. Lastly, constraint \eqref{constr:demand} ensures that the uncertain demand is satisfied at each store.

Regarding the optimization problem data, we consider a contrived example with $|\set{F}| = 5$ factories at locations that are centrally located in the United States: Des Moines, Iowa; Kansas City, Missouri; Denver, Colorado; Wichita, Kansas; and St. Louis, Missouri. We consider $|\set{H}| = 7$ warehouses in the following cities: Portland, Oregan; Salt Lake City, Utah; Phoenix, Arizona; Charlotte, North Carolina; Atlanta, Georgia; Cincinnati, Ohio; and Chicago, Illinois. Lastly, we consider $|\set{S}| = 5$ stores in larger metropolitan areas: Dallas, Texas; Los Angeles, California; New York, New York; Orlando, Florida; and Seattle, Washington. Hence the network optimization problem \eqref{prob:network} has a total of 75 variables and 24 constraints. We compute the values $c^1$ and $c^2$ using the distance between the respective cities. Namely, we obtain the distance in kilometers using the dataset provided in \citep{erickson_county_city_driving_2014} and divide by 1000. We set the parameter $\beta = 10$. Lastly, we set the capacity parameter $M$ according to a real-world historical dataset.

\textit{Context Data}: Based on the historical sales data of a company and their distribution network, we synthetically generate a larger network to include major cities in the United States, described in detail above. For the contextual features, we use average daily temperature from each city where a store is located, the day of the week, and the month. Because of the sparsity of the weekend data, we only consider Monday through Friday. We convert the categorical ``day of the week'' and ``month'' features are to numeric features by one-hot encoding \citep{bishop2006pattern}. The result is a context vector $\bs{\obs}_i \in \RR^{21}$, where the first feature is unity for an intercept term, the next 5 features are the average temperature in each of the 5 cities corresponding to the store locations, the next 11 features correspond to the month, and the last 4 correspond to the day of the week. Associated with this is a vector $\bs{b}_i \in \RR^5$, i.e., one sales/demand observation for each city. We note that the linear model $\bs{W} \in \RR^{5 \times 21}$.

\textit{Learning Problems}: Because of the structure of the network optimization problem and the context data, we must slightly modify the learning problems. To do this, we define the submatrix $\bs{A}^=$ of the constraint matrix $\bs{A}$ generated by problem \eqref{prob:network} corresponding to the equality constraints, and similarly for $\bs{A}^{\leq}$ and $\bs{A}^{\geq}$. We also consider the associated subvectors $\bs{b}^=, \bs{b}^{\leq}$, and $\tilde{\bs{b}}^{\geq}$, and their respective dual vectors $\bs{y}^=, \bs{y}^{\leq}$, and $\bs{y}^{\geq}$. Observe that we are only predicting components for the uncertain subvector $\tilde{\bs{b}}^{\geq}$. Additionally, we want to enforce some of the components of the model $\bs{W}$ to be equal to 0. Take for example the component $W_{13}$. This component corresponds to the prediction of demand in store \#1 since it is in the first row of $\bs{W}$. However, the inner product $\inner{\bs{w}_1, \bs{\obs}}$ contains the term $W_{13}\obs_{3}$, where $\obs_{3}$ corresponds to a realization of the average daily temperature corresponding to store \#2 (recall that the first component of $\bs{\obs}$ is unity). That is, we do not want temperature data from one store to affect the prediction of demand in another store. We let $\set{W}^0 \coloneqq \{(j, k) \in [5] \times [6]\setminus\{1\} ~ | ~ k \neq j + 1\}$ be the set of indices for which the correpsponding component of $\bs{W}$ is set to 0. These indices correspond to the off-diagonal elements of the $5 \times 5$ submatrix corresponding to the temperature features, and is directly to the right of the first column of $\bs{W}$ (the intercept column).

We update the optimistic-DAL problem \eqref{prob:optimisticTrainLinear} as
\begin{align}\label{prob:optimisticTrainLinear_network}
\min_{\bs{W}} \quad & \left(\inner{\bs{c} , \bs{x}_{i}^{\star}} -  \left(\inner{\bs{b}_i^=, (\bs{y}_i^=)^\star} + \inner{\bs{b}_i^\leq, (\bs{y}_i^\leq)^\star} + \inner{\bs{W\obs}_i, (\bs{y}_i^\geq)^\star}\right)\right) \nonumber \\
\text{s.t.} \quad & \bs{A}^\geq\bs{x}_{i}^{\star} \geq \bs{W\obs}_{i}, ~ \forall i \in [N], \nonumber \\
\quad & W_{jk} = 0, ~ \forall (j, k) \in \set{W}^0.
\end{align}
We update the primal-DAL problem \eqref{prob:fixedPrimalTrainLinear} as
\begin{align}\label{prob:fixedPrimalTrainLinear_network}
\min_{\bs{W}, (\bs{y}_i)} \quad & F(\, \bs{W}, (\bs{y}_i)\,) \nonumber \\
\text{s.t.} \quad & \bs{A}^\geq\bs{x}_{i}^{\star} \geq \bs{W\obs}_{i}, ~ \forall i \in [N], \nonumber \\
\quad & (\bs{A}^{=})^\top \bs{y}_i^= + (\bs{A}^\leq)^\top \bs{y}_i^\leq + (\bs{A}^{\geq})^\top \bs{y}_i^\geq \leq \bs{c}, ~ \forall i \in [N] \nonumber \\
\quad & \bs{y}^\leq \leq \bs{0}, \nonumber \\
\quad & \bs{y}^\geq \geq \bs{0}, \nonumber \\
\quad & W_{jk} = 0, ~ \forall (j, k) \in \set{W}^0.
\end{align}
where the objective function is defined as 
\begin{align*}
    F( \, \bs{W}, (\bs{y}_i) \,) = \frac{1}{N} \sum_{i \in [N]} \left(\inner{\bs{c} , \bs{x}_{i}^{\star}} -  \left(\inner{\bs{b}_i^=, \bs{y}_i^=} + \inner{\bs{b}_i^\leq, \bs{y}_i^\leq} + \inner{\bs{W\obs}_i, \bs{y}_i^\geq}\right)\right) + \lambda \, r(\bs{W}) + \gamma \, \phi(\bs{W}).
\end{align*}
The dual-DAL problem \eqref{prob:fixedDualModTrain} becomes
\begin{align}\label{prob:fixedDualModTrain_network}
\min_{\bs{W}, (\bs{x}_i)} \quad & \frac{1}{N}\sum_{i \in [N]}\left(\inner{\bs{c} , \bs{x}_{i}} -  \left(\inner{\bs{b}_i^=, (\bs{y}_i^=)^\star} + \inner{\bs{b}_i^\leq, (\bs{y}_i^\leq)^\star} + \inner{\alpha\bs{W\obs}_i - \bs{b}_i, (\bs{y}_i^\geq)^\star}\right)\right) \nonumber \\
\text{s.t.} \quad & \bs{A}^\geq\bs{x}_{i} \geq \alpha\bs{W\obs}_{i} - \bs{b}_i^\geq, ~ \forall i \in [N], \nonumber \\
\quad & \bs{A}^=\bs{x}_{i} = (\alpha - 1)\bs{b}_i^=, ~ \forall i \in [N], \nonumber \\
\quad & \bs{A}^\leq\bs{x}_{i} \leq (\alpha - 1)\bs{b}_i^\leq, ~ \forall i \in [N], \nonumber \\
\quad & \bs{x}_i \geq 0, ~ \forall i \in [N], \nonumber \\
\quad & W_{jk} = 0, ~ \forall (j, k) \in \set{W}^0.
\end{align}
We see that problem \eqref{prob:fixedDualModTrain_network} perturbs the right-hand side values corresponding to the constraints for which we are not generating predictions ($\bs{b}_i^=$ and $\bs{b}_i^\leq$). Hence the only choice for that makes sense is $\alpha = 2$. Regarding the DUL models, we solve the linear regression problem
\begin{align}
\label{prob:linearRegression_network}
\min_{\bs{W} \in \mathbb{R}^{5 \times 21}} \quad & ||\mathfrak{X} W^\top - \mathfrak{B}||_F^2 \nonumber \\
\text{s.t.} \quad & W_{jk} = 0, ~ \forall (j, k) \in \set{W}^0.
\end{align}
where
$\mathfrak{X} = \begin{bmatrix}
\bs{\obs}_{1}^\top\\
\vdots\\
\bs{\obs}_{N}^\top
\end{bmatrix} \in \RR^{N \times 21}$,
$\mathfrak{B} = \begin{bmatrix}
\bs{b}_{1}^\top\\
\vdots\\
\bs{b}_{N}^\top
\end{bmatrix} \in \RR^{N \times 5}$, and $N$ is the number of training datapoints. We also solve the lasso regression problem
\begin{align}
\label{prob:lassoRegression_network}
\min_{\bs{W} \in \mathbb{R}^{5 \times 21}} \quad & ||\mathfrak{X} W^\top - \mathfrak{B}||_F^2 + \alpha_{lasso}\sum_{j \in [m]}\sum_{k \in [d]}|W_{jk}| \nonumber \\
\text{s.t.} \quad & W_{jk} = 0, ~ \forall (j, k) \in \set{W}^0.
\end{align}

\textit{Hyperparameter Tuning}: For the primal-DAL problem, we do not tune with the feasibility metric $\chi\{\bs{A}^\geq\bs{x}_i^\star \geq \bs{W\obs}_i\}$ as was done in the synthetic experiment in \S\ref{subsec:synthetic_data_experiments}. This is because the zero matrix $\bs{W} \equiv \bs{0}$ is feasible to the constraints $\bs{A}^\geq\bs{x}_i^\star \geq \bs{W\obs}_i$ in the network flow problem \eqref{prob:network} and we want to discourage this problem from producing such a model. Instead, we utilize the predicted optimality gap metric $\inner{\bs{c}, \bs{x}_i^\star} - \inner{p(\bs{\obs}_i^v), \bs{y}_i^\star}$ (lower is better), which we compute only for datapoints such that $\bs{A}^\geq\bs{x}_i^\star \geq \bs{W\obs}_i$. In the absence of additional constraints on the model $\bs{W}$, we tune with candidates $(\lambda, \gamma) \in \{10^{-12}, 10^{-9}, 10^{-6}, 10^{-3}, 10^0, 10^3\} \times \{0\}$ and in the presence of the additional constraints $\bs{W\obs}_i \geq \bs{b}_i^\geq, ~ \forall i \in [N],$ we tune with candidates $(\lambda, \gamma) \in \{10^{-12}, 10^{-6}, 10^0, 10^6 \}^2$.
Unlike the synthetic data experiments, we set $\alpha = 2$ in the dual-DAL problem instead of tuning this parameter. The reason for this is because the network optimization problem we are considering contains constraints whose right-hand side value we are not predicting (see earlier in Appendix \S\ref{sec:network_optimization_details} for more details). Finally, we tune $\alpha_{lasso}$ in the lasso regression problem using the sum of square prediction errors as the metric, with candidate values $\{1, 3, 5, 7\}$.

\end{document}